\documentclass{amsart}
\usepackage{amsmath,amssymb,amsfonts,amsthm}
\usepackage{eucal}
\usepackage{stmaryrd}
\usepackage[normalem]{ulem}
\usepackage{verbatim}
\usepackage{txfonts} 
\usepackage{comment}

\usepackage{enumerate}
\newenvironment{enumroman}{\begin{enumerate}[\upshape (i)]}{\end{enumerate}}

\newenvironment{myindentpar}[1]
{\begin{list}{}
         {\setlength{\leftmargin}{#1}}
         \item[]
}
{\end{list}}

\input xy
\xyoption{all}
\newdir{ >}{{}*!/-10pt/@{>}}
\newdir{ |}{{}*!/-5pt/@{|}}
\newdir{> }{{}*!/5pt/@{>}}

\theoremstyle{plain}
\newtheorem{thm}{Theorem}[section] 
\newtheorem{cor}[thm]{Corollary}

\newtheorem{lem}[thm]{Lemma}

\theoremstyle{definition}
\newtheorem{defn}[thm]{Definition}

\newtheorem{ntn}[thm]{Notation}
\newtheorem*{thank}{Acknowledgments}

\theoremstyle{remark}
\newtheorem{ex}[thm]{Example}
\newtheorem{rmk}[thm]{Remark}

\DeclareMathOperator*{\pb}{{\times}}
\DeclareMathOperator*{\po}{{\sqcup}}

\makeatletter
\let\c@equation\c@thm
\makeatother
\numberwithin{equation}{section}

\newcommand{\R}{\mathbb{R}}

\renewcommand{\L}{\mathbb{L}}
\newcommand{\G}{\mathcal{G}}
\newcommand{\W}{\mathcal{W}}
\newcommand{\M}{\mathcal{M}}
\newcommand{\N}{\mathcal{N}}
\newcommand{\K}{\mathcal{K}}
\newcommand{\I}{\mathcal{I}}
\newcommand{\J}{\mathcal{J}}
\newcommand{\V}{\mathcal{V}}

\newcommand{\cL}{\mathcal{L}}
\newcommand{\cR}{\mathcal{R}}
\newcommand{\cC}{\mathcal{C}}
\newcommand{\cF}{\mathcal{F}}

\newcommand{\D}{\mathbb{D}}

\newcommand{\C}{\mathbb{C}}
\newcommand{\F}{\mathbb{F}}

\renewcommand{\a}{\alpha}
\renewcommand{\b}{\beta}
\renewcommand{\d}{\delta}
\newcommand{\e}{\epsilon}
\renewcommand{\l}{\lambda}

\renewcommand{\r}{\rho}

\newcommand{\iso}{\mathbf{I}}

\newcommand{\op}{\mathrm{op}}

\newcommand{\id}{\text{id}}

\newcommand{\dom}{\mathrm{dom}}
\newcommand{\cod}{\mathrm{cod}}

\renewcommand{\lim}{\text{lim}}

\newcommand{\hast}{\hat{\otimes}}
\newcommand{\htimes}{\hat{\times}}
\newcommand{\hotimes}{\hat{\otimes}}
\newcommand{\hsmash}{\hat{\wedge}}
\newcommand{\hhom}{\hat{\mathrm{hom}}}
\newcommand{\hhoml}{\hat{\mathrm{hom}_\ell}}
\newcommand{\hhomr}{\hat{\mathrm{hom}_r}}
\newcommand{\homr}{\mathrm{hom}_r}
\newcommand{\homl}{\mathrm{hom}_\ell}

\renewcommand{\1}{{\boldsymbol{1}}}
\newcommand{\4}{{\boldsymbol{4}}}
\renewcommand{\2}{{\boldsymbol{2}}}
\newcommand{\3}{{\boldsymbol{3}}}

\newcommand{\Coalg}{\mathbb{C}\mathrm{\bf oalg}}
\newcommand{\Alg}{\mathbb{A}\mathrm{\bf lg}}
\newcommand{\SSq}{\mathbb{S}\mathrm{\bf q}}

\newcommand{\ladj}{\mathrm{ladj}}

\newcommand{\Rtalg}{\mathbb{R}_t\text{-}\mathrm{{\bf alg}}}
\newcommand{\Rpalg}{\mathbb{R}'\text{-}\mathrm{{\bf alg}}}
\newcommand{\Ralg}{\mathbb{R}\text{-}\mathrm{{\bf alg}}}
\newcommand{\Ftalg}{\mathbb{F}_t\text{-}\mathrm{{\bf alg}}}
\newcommand{\Falg}{\mathbb{F}\text{-}\mathrm{{\bf alg}}}
\newcommand{\Fpalg}{\mathbb{F}'\text{-}\mathrm{{\bf alg}}}

\newcommand{\Qalg}{\mathbb{Q}\text{-}\mathrm{{\bf coalg}}}
\newcommand{\Ltalg}{\mathbb{L}_t\text{-}\mathrm{{\bf coalg}}}
\newcommand{\Lalg}{\mathbb{L}\text{-}\mathrm{{\bf coalg}}}
\newcommand{\Ctalg}{\mathbb{C}_t\text{-}\mathrm{{\bf coalg}}}
\newcommand{\Calg}{\mathbb{C}\text{-}\mathrm{{\bf coalg}}}
\newcommand{\Cpalg}{\mathbb{C}'\text{-}\mathrm{{\bf coalg}}}
\newcommand{\Cptalg}{\mathbb{C}'_t\text{-}\mathrm{{\bf coalg}}}
\newcommand{\Lpalg}{\mathbb{L}'\text{-}\mathrm{{\bf coalg}}}

\newcommand{\lra}{\longrightarrow}
\newcommand{\sr}{\stackrel}
\newcommand{\To}{\Rightarrow}

\newcommand{\Cat}{\text{\bf Cat}}

\renewcommand{\Top}{\text{\bf Top}}
\newcommand{\sSet}{\text{\bf sSet}}

\newcommand{\CAT}{\text{\bf CAT}}
\newcommand{\FunF}{\text{\bf FF}}
\newcommand{\AWFS}{\text{\bf AWFS}}
\newcommand{\LAWFS}{\text{\bf LAWFS}}
\newcommand{\Cmd}{\text{\bf Cmd}}

\begin{document}

\title{Monoidal algebraic model structures}
\author{Emily Riehl}

\address{Deptartment of Mathematics\\Harvard University \\ 1 Oxford Street\\ Cambridge, MA 02138}
\email{ eriehl@math.harvard.edu}
\date{\today}

\maketitle

\begin{abstract} 
Extending previous work, we define monoidal algebraic model structures and give examples. The main structural component is what we call an algebraic Quillen two-variable adjunction; the principal technical work is to develop the category theory necessary to characterize them. Our investigations reveal an important role played by ``cellularity''---loosely, the property of a cofibration being a relative cell complex, not simply a retract of such---which we particularly emphasize. A main result is a simple criterion which shows that algebraic Quillen two-variable adjunctions correspond precisely to cell structures on the pushout-products of generating (trivial) cofibrations. As a corollary, we discover that the familiar monoidal model structures on categories and simplicial sets admit this extra algebraic structure.
\end{abstract}

\tableofcontents 

\section{Introduction}

Algebraic model structures, introduced in \cite{riehlalgebraic}, are a structural extension of Quillen's model categories \cite{quillenhomotopical} in which cofibrations and fibrations are ``algebraic,'' i.e., equipped with specified retractions to their left or right factors which can be used to solve all lifting problems. The factorizations themselves are much more than functorial: the functor mapping an arrow to its right factor is a monad and the functor mapping to its left factor is a comonad on the arrow category. In particular, the data of an algebraic model category determines a fibrant replacement monad and a cofibrant replacement comonad.

Despite the stringent structural requirements of this definition, algebraic model structures are quite abundant. A modified small object argument, due to Richard Garner produces an algebraic model structure in place of an ordinary cofibrantly generated one \cite{garnerunderstanding}. The difference is that the main components of a model structure {---}the \emph{weak factorization systems} $(\cC \cap \W, \cF)$ and $(\cC, \cF \cap \W)$---are replaced with \emph{algebraic weak factorization systems} $(\C_t,\F)$ and $(\C,\F_t)$, which are categorically better behaved.

We find the weak factorization system perspective on model categories clarifying. The overdetermination of the model category axioms and the closure properties of the classes of cofibrations and fibrations are consequences of analogous characteristics of the constituent weak factorization systems. Quillen's small object argument is really a construction of a functorial factorization for a cofibrantly generated weak factorization system; the model structure context is beside the point. Also, the equivalence of various definitions of a Quillen adjunction has to do with the separate interactions between the adjunction and each weak factorization system. 

More precisely, an \emph{algebraic model structure} on a category $\M$ with weak equivalences $\W$ consists of two algebraic weak factorization systems (henceforth, \emph{awfs} for both the singular and the plural) together with a morphism $\xi \colon (\C_t,\F) \to (\C,\F_t)$ between them such that the underlying weak factorization systems form a model structure in the usual sense. Here $\C_t$ and $\C$ are comonads and $\F_t$ and $\F$ are monads on the arrow category $\M^\2$ that send an arrow to its appropriate factor with respect to the functorial factorizations of the model structure. We write $R,Q \colon \M^\2 \rightrightarrows \M$ for the functors that assign to an arrow the object through which it factors. The notation is meant to evoke fibrant/cofibrant replacement: slicing over the terminal object or under the initial object defines the fibrant replacement monad and cofibrant replacement comonad, also denoted $R$ and $Q$.

The natural transformation $\xi$, which we call the \emph{comparison map}, plays a number of roles. Its components \begin{equation}\label{xidefn} \xymatrix{ \dom f \ar[d]_{C_tf} \ar[r]^{Cf} & Qf \ar[d]^{F_tf} \\ Rf \ar[r]_{Ff} \ar[ur]^{\xi_f} & \cod f}\end{equation}
are natural solutions to the lifting problem (\ref{xidefn}) that compares the two functorial factorizations of $f \in \M^\2$. Additionally, $\xi$ must satisfy two pentagons: one involving the comultiplications of the comonads and one involving the multiplications of the monads. Under these hypotheses, $\xi$ determines functors over $\M^\2$ \begin{equation}\label{xifunctors} \xi_* \colon  \Ctalg \to \Calg \quad \quad\quad  \xi^* \colon \Ftalg \to \Falg \end{equation} between the categories of coalgebras for the comonads and between the categories of algebras for the monads.

Elements of, e.g., the category $\Falg$ are called \emph{algebraic fibrations}; their images under the forgetful functor to $\M^\2$ are in particular fibrations in the model structure. The algebra structure associated to an algebraic fibration determines a canonical solution to any lifting problem of that arrow against an algebraic trivial cofibration. Naturality of $\xi$ together with the functors (\ref{xifunctors}) imply that there is also a single canonical solution to any lifting problem of an algebraic trivial cofibration against an algebraic trivial fibration: the solution constructed using $\xi_*$ and the awfs $(\C,\F_t)$ agrees with the solution constructed using $\xi^*$ and the awfs $(\C_t,\F)$.

For certain lifting problems, these canonical solutions themselves assemble into a natural transformation. For instance, the natural solution to the usual lifting problem that compares the two fibrant-cofibrant replacements of an object defines a natural transformation $RQ \To QR$ that turns out to be a distributive law of the fibrant replacement monad over the cofibrant replacement comonad. It follows that $Q$ lifts to a comonad on the category $\Ralg$ of algebraic fibrant objects, and dually $R$ lifts to a monad on $\Qalg$. The coalgebras for the former and algebras for the later coincide, defining a category of algebraic fibrant-cofibrant objects.

Any ordinary cofibrantly generated model structure gives rise to an algebraic model structure using a modified form of Quillen's small object argument due to Richard Garner. As a result, this algebraic structure is much more common that might be supposed. Whenever the category permits the small object argument, any small category of arrows generates an algebraic weak factorization system that satisfies two universal properties, both of which we frequently exploit \cite{garnercofibrantly, garnerunderstanding}. 

Awfs were introduced to improve the categorical properties of ordinary weak factorization systems \cite{gtnatural}. One feature of awfs is that the left and right classes are closed under colimits and limits, respectively, in the following precise sense. By standard monadicity results, the forgetful functors $\Calg \to \M^\2$, $\Falg \to \M^\2$ create all colimits and limits, respectively, existing in $\M^\2$. In the context of algebraic model structures, this gives a new recognition principle for cofibrations constructed as colimits and fibrations constructed as limits.  Familiarly, a colimit (in the arrow category) of cofibrations is not necessarily a cofibration. But if the cofibrations admit coalgebra structures that are preserved by the maps in the diagram, then the colimit is canonically a coalgebra and hence a cofibration.

When the model structure is cofibrantly generated, all fibrations and all trivial fibrations are \emph{algebraic}, i.e., admit algebra structures; interestingly the dual statements do not hold. Transfinite composites of pushouts of coproducts of generating cofibrations $i \in \I$---the class of maps denoted $\I$-cell in the classical literature \cite{hoveymodel, hirschhornmodel}---are necessarily algebraic cofibrations. Accordingly, we call the class of cofibrations that admit a $\C$-coalgebra structure the \emph{cellular} cofibrations; a cofibration is cellular if and only if it can be made algebraic. All cofibrations are at least retracts of cellular ones. Cellularity will play an interesting and important role in the new results that follow.

For example, if $A$ is a commutative ring, the arrow $0 \to A$ generates an awfs on the category of $A$-modules whose right class is the epimorphisms and whose left class is the injections with \emph{projective} cokernel. Each epimorphism $M \twoheadrightarrow N$ admits (likely many) algebra structures: an algebra structure is a section $N \to M$, not assumed to be a homomorphism. The cellular maps---that is, those arrows admitting coalgebra structures---are those injections which have \emph{free} cokernel.

The basic theory of algebraic model structures is developed in \cite{riehlalgebraic}; references to results therein will have the form I.x.x. In particular, that paper defines an \emph{algebraic Quillen adjunction}, which is an ordinary Quillen adjunction such that the right adjoint lifts to commuting functors between the algebraic (trivial) fibrations and the left adjoint lifts to commuting functors between the categories of algebraic (trivial) cofibrations. This should be thought of an algebraization of the usual condition that the right adjoint preserves fibrations and trivial fibrations and left adjoint preserves cofibrations and trivial cofibrations. We also ask that the lifts of one adjoint determine the lifts of the other in a sense made precise below, a condition that mirrors the classical fact that a Quillen adjunction can be detected by examining the left or right adjoint alone. Algebraic Quillen adjunctions exist in an important class of examples: when a cofibrantly generated algebraic model structure is lifted along an adjunction, the resulting Quillen adjunction is canonically algebraic. Examples include the geometric realization--total singular complex adjunction between simplicial sets and spaces, the adjunction between $G$-spaces and space-valued presheaves on the orbit category for a group $G$, and the adjunctions establishing a projective model structure.

A classical categorical result characterizes lifted functors of algebraic fibrations, i.e., functors between the categories of algebras for the monads, as certain natural transformations sometimes called \emph{lax monad morphisms}, but this condition alone fails to capture the symmetry of the classical situation where a right adjoint preserves fibrations if and only if its left adjoint preserves trivial cofibrations. There are two ways to describe the desired additional hypothesis. One is to ask that the \emph{mate} of the natural transformation characterizing the lifted functor of algebraic fibrations defines the lifted functor of algebraic trivial cofibrations. An equivalent condition is that the lifted functor of algebraic fibrations is in fact a lifted \emph{double functor} between double categories of algebraic fibrations, suitably defined. 

In this paper, we extend these results in order to define monoidal and eventually enriched algebraic model structures. The technical work in this paper puts the latter definition in immediate reach; however, we postpone it to a future paper which will have space to fully explore examples.  Much of the structure of a closed monoidal category or a tensored and cotensored enriched category is encoded in a \emph{two-variable adjunction}. For enriched categories, the constituent bifunctors are commonly denoted \[\xymatrix@C=30pt{ \V \times \M \ar[r]^-{-\odot-} & \M} \quad \xymatrix@C=30pt{ \V^\op \times \M \ar[r]^-{ \{-,-\}} &\M} \quad \xymatrix@C=30pt{ \M^\op \times \M \ar[r]^-{\hom(-,-)} & \V}\] and come equipped with hom-set isomorphisms \begin{equation}\label{enrichedhomset} \M(V \odot M, N) \cong \M(M, \{V,N\}) \cong \V(V, \hom(M,N))\end{equation} natural in all three variables. Fixing any one variable, two-variable adjunctions give rise to parameterized families of ordinary adjunctions, e.g., $- \odot M \dashv \hom(M,-)$.

The monoidal case necessarily precedes the enriched one but also inherits  all of its complexity.
A closed monoidal category with an algebraic model structure is a \emph{monoidal algebraic model category} if the canonical comparison between a cofibrant object and its tensor with the cofibrant replacement of the monoidal unit is a weak equivalence and if the closed monoidal structure is an \emph{algebraic Quillen two-variable adjunction}. Such an adjunction consists of three functors 
\[\xymatrix@R=0pt{\Ctalg\times \Calg \ar[r] & \Ctalg \\ \Calg \times \Ctalg \ar[r] & \Ctalg \\ \Calg \times \Calg \ar[r] & \Calg}\]
lifting the so-called ``pushout-product'' such that the mates of the characterizing natural transformations determine similar lifts of the left and right closures. In the best cases, these functors satisfy three evident coherence conditions which say that various canonical coalgebra structures agree, but we shall see that such coherence is too much to ask for in general. 

One could also define a weaker notion of an algebraic Quillen bifunctor in the context of monoidal or enriched model categories in which some of the adjoint bifunctors don't exist. This is  less categorically challenging than the theory presented here, so the details may be safely left to the reader. 

Three main technical theorems allow us to identify algebraic Quillen two-variable adjunctions in practice. The first describes a composition criterion that identifies when a lifted bifunctor is part of a two-variable adjunction of awfs, the appropriate notion of algebraic Quillen two-variable adjunction for categories equipped with a single awfs in place of a full algebraic model structure. The other two results, which we call the cellularity and uniqueness theorems, combine to characterize two-variable adjunctions of awfs in the case when the awfs are cofibrantly generated. The cellularity theorem says that a two-variable adjunction of awfs arises from any assignment of coalgebra structures to the pushout-product of the generators; hence, such structures exist if and only if the pushout-product of the generators is \emph{cellular}. The uniqueness theorem says that such an assignment completely determines the lifted functors, so at most one two-variable adjunction of awfs can be obtained in this way. 

Several new categorical results were necessary to make all of this precise. Of most general categorical interest is the theory of \emph{parameterized mates}, introduced in \S\ref{matessec} below. This theory describes the relationship between the natural transformations characterizing the lifts of the three bifunctors constituting a two-variable adjunction of awfs and their interactions with ordinary adjunctions of awfs.

Other results appearing below are designed to deal with complications arising in the proofs of the cellularity and uniqueness theorems. The main technical difficulty is quite simply accounted for: in \cite{riehlalgebraic}, the only adjunctions considered between arrow categories were those of the form $T^\2 \colon \M^\2 \rightleftarrows \K^\2\colon  S^\2 $, i.e., defined pointwise by an ordinary adjunction $T  \colon \M \rightleftarrows \K\colon S$ between the base categories. However, the adjunctions on arrow categories arising from two-variable adjunctions on the bases no longer have this form and in particular don't preserve composability of arrows. Thus, the double categorical composition criterion we use to great effect in the previous paper to characterize those lifted left adjoints that determine lifts of right adjoints must take on a new form.

In \S\ref{matessec}, we introduce double categories, mates, and parameterized mates and prove some elementary lemmas which will be quoted frequently. In \S\ref{prelimsec}, we review lifting properties and functorial factorizations, wfs and awfs, and the algebraic small object argument. In \S\ref{morawfssec}, we present a variety of notions of morphism between awfs on different categories and define the new notion of two-variable adjunction of awfs. In \S\ref{compsec}, we prove the composition criterion which allows us to recognize when a given lifted bifunctor of awfs (co)algebras determines a two-variable adjunctions of awfs. In \S\ref{cellsec}, we use this result to prove the cellularity theorem. We then extend the universal property of Garner's small object argument and use this to prove the uniqueness theorem. In \S\ref{algQuillsec}, we apply these results to model categories, introducing a notion of algebraic Quillen two-variable adjunctions. Finally in \S\ref{monsec}, we define monoidal algebraic model structures and describe examples.

\section{Double categories, mates, parameterized mates}\label{matessec}

The calculus of \emph{mates} will play an important conceptual and calculational role in what follows. To streamline later proofs, we take a few moments in \S\ref{matesssec} to outline the important features without getting mired in technical details. The canonical reference is \cite{kellystreetreview}; we also like \cite{shulmancomparing}.

Bifunctors, meaning functors whose domain is the product of two categories, are determined by the collection of single-variable functors obtained when one object is fixed together with the natural transformations between such functors arising from morphisms in that category. This fact is often expressed by saying that category $\CAT$ is cartesian closed. For this simple reason, the classical theory of mates extends to a new theory of \emph{parameterized mates}, introduced in \S\ref{paramatesssec}.

\subsection{Double categories and mates}\label{matesssec}

A \emph{double category} $\D$ is a category internal to $\CAT$:
\[\xymatrix{ \D_1 \displaystyle\pb_{\D_0} \D_1 \ar[r]^-{\circ} & \D_1 \ar@<1ex>[r]^{\dom} \ar@<-1ex>[r]_{\cod} & \D_0 \ar[l]|{\id}}\] The objects and arrows of $\D_0$ are called objects and horizontal arrows of $\D$ while the objects and arrows of $\D_1$ are called vertical arrows and squares. Via the functors $\dom,\cod \colon \D_1 \rightrightarrows \D_0$, the sources and targets of vertical arrows are objects of $\D$, and likewise the squares can be depicted in the way their name suggests. Squares can be composed horizontally using composition in $\D_1$ and vertically using the functor $\circ$, whose domain is the pullback of dom along cod. As a consequence of functoriality of $\circ$, the order in which vertical and horizontal composites are taken in a pasting diagram of squares does not matter.  We refer to $\D_1$ as the \emph{category of vertical arrows}; this category forgets the composition of vertical arrows and remembers only the horizontal composition of squares.

\begin{ex} 
A category $\M$ gives rise to a double category $\SSq(\M)$ 
\[\xymatrix{ \M^\3 \cong \M^\2 \displaystyle\pb_{\M} \M^\2 \ar[r]^-{\circ} & \M^\2 \ar@<1ex>[r]^{\dom} \ar@<-1ex>[r]_{\cod} & \M \ar[l]|{\id}}\]
whose objects are objects of $\M$, horizontal and vertical arrows are morphisms of $\M$, and squares are commutative squares. The category of vertical arrows is usually called the arrow category and plays an essential role in this paper.
\end{ex}

Given categories, functors, and adjunctions, as displayed below, there is a bijection between natural transformations in the square involving the left adjoints and natural transformations in the square involving the right adjoints
\begin{equation}\label{genericmates}\xymatrix{\cdot \ar@<-.75ex>[d]_-T \ar@{}[d]|-{\dashv} \ar[r]^H & \cdot\ar@<-.75ex>[d]_-{T'} \ar@{}[d]|-{\dashv} &  &  \cdot \ar[d]_T \ar[r]^H \ar@{}[dr]|{\l\, \Swarrow} & \cdot \ar[d]^{T'} & \ar@{}[d]|{\leftrightsquigarrow} & \cdot \ar[r]^H \ar@{}[dr]|{\Searrow\ \r}  & \cdot \\ \cdot \ar@<-.75ex>[u]_-S \ar[r]_K & \cdot \ar@<-.75ex>[u]_-{S'} & & \cdot \ar[r]_K & \cdot  &  & \cdot \ar[u]^S \ar[r]_K & \cdot \ar[u]_{S'}}\end{equation} 
given by the formulas
\begin{equation}\label{matesformula} \r = S'K\e \cdot S'\l_S \cdot \iota_{HS} \quad \mathrm{and} \quad \l = \nu_{KT} \cdot T'\r_T \cdot T'H\eta, \end{equation}
where $\eta$ and $\e$ are the unit and counit for $T \dashv S$ and $\iota$ and $\nu$ are the unit and counit for $T' \dashv S'$. Corresponding $\l$ and $\r$ are called \emph{mates}. 

\begin{ex}\label{trivialmatesex}
A natural transformation $H \To K$ is its own mate with respect to the identity adjunctions.
\end{ex}

\begin{ex} Write $\1$ for the terminal category. \emph{Adjunct} arrows $f^\sharp \colon Tm \to k \in \K$, $f\colon m \to Sk \in \M$ corresponding under the adjunction $T \colon \M \rightleftarrows \K\colon S$ are mates in the following squares \[\xymatrix{ {\1} \ar[d]_1 \ar[r]^m \ar@{}[dr]|{f^\sharp \Swarrow} & \M \ar[d]^T & \ar@{}[d]|{\leftrightsquigarrow} & {\1} \ar[r]^m \ar@{}[dr]|{\Searrow f} & \M  \\ {\1} \ar[r]_k & \K &&{\1} \ar[r]_k \ar[u]^1 & \K \ar[u]_S} \] 
\end{ex}

\begin{ex}\label{paramatesex} If $\M$ has a left-closed monoidal structure and $f \colon m' \to m \in \M$, then the induced natural transformations \[\xymatrix@C=35pt{ \M \ar[r]^1 \ar[d]_{m \otimes -} \ar@{}[dr]|{f \otimes- \Swarrow} & \M \ar[d]^{m' \otimes -} & \ar@{}[d]_{\leftrightsquigarrow}  &  \M \ar[r]^1 \ar@{}[dr]|{\Searrow \hom(f,-)} & \M \\ \M \ar[r]_1 & \M && \M \ar[r]_1 \ar[u]^{\hom(m,-)} & \M \ar[u]_{\hom(m',-)} }\] are mates. Analogous correspondences hold for any \emph{parameterized adjunction} \cite[IV.7.3]{maclanecategories}.
\end{ex}

There are double categories $\L${\bf adj} and $\R${\bf adj} whose objects are categories, horizontal arrows are functors, vertical arrows are adjunctions in the direction of the left adjoint, and whose squares are natural transformations as displayed in the middle and right-hand squares of (\ref{genericmates}), respectively. The mates correspondence is natural, or, more accurately, functorial, in the following precise sense.

\begin{thm}[Kelly-Street {\cite[\S 2]{kellystreetreview}}]\label{dblmatesthm}The mates correspondence gives an isomorphism of double categories $\L${\bf adj}$ \cong \R${\bf adj}.
\end{thm}

This says that a natural transformation obtained by pasting squares in $\L${\bf adj} either vertically or horizontally is the mate of the natural transformation obtained by pasting the mates of these squares in $\R${\bf adj}. The ``calculus of mates'' refers to this fact, which, when used in conjunction with Examples \ref{trivialmatesex}--\ref{paramatesex}, implies that mates satisfy ``dual'' diagrams. 

For instance, suppose the functors $H$ and $K$ of (\ref{genericmates}) are monads $(H,\eta,\mu)$, $(K, \eta,\mu)$ and suppose $T=T'$ and $S = S'$. A pair $(S,\rho)$ as in the right square of (\ref{genericmates}) is a \emph{lax monad morphism} if
\begin{equation}\label{laxmor}\vcenter{\xymatrix@C=30pt@R=20pt@!0{ & S \ar[ddl]_{{\eta}_S} \ar[ddr]^{S\eta}\\ \\ HS \ar[rr]^{{\rho}} & & SK }}\hspace{.5cm} \mathrm{and} \hspace{.5cm}  \vcenter{\xymatrix@C=20pt@R=20pt@!0{  & & HSK \ar[drr]^{{\rho}_K} & & \\ HHS \ar[urr]^{H{\rho}} \ar[ddr]_{{\mu}_S} & & & & SKK \ar[ddl]^{S{\mu}} \\ \\ & HS \ar[rr]^{{\rho}} & & SK &}}\end{equation} commute. The definitions in \S\ref{morawfssec} take several equivalent forms on account of the following result.

\begin{lem}[{Appelgate \cite{johnstoneadjoint}}]\label{appellem}
A lax monad morphism $(S,\rho)$ determines and is determined by a lift of $S$ to a functor $\mathbb{K}$-{\bf alg}$\to \mathbb{H}$-{\bf alg}.
\end{lem}
\begin{proof}
The $\mathbb{H}$-algebra structure assigned to the image of a $\mathbb{K}$-algebra $t \colon Kx \to x$ under $S$ is \[\xymatrix{ HSx \ar[r]^{\r_x} & SKx \ar[r]^{St} & Sx} \qedhere\]
\end{proof}

The dual notion, a \emph{colax monad morphism}, is a pair $(S,\rho)$ satisfying diagrams analogous to (\ref{laxmor}) but with the direction of $\r$ reversed. 

\begin{lem}\label{mateslem} Suppose $(S,\rho)$ is a lax monad morphism, $T \dashv S$, and $\l$ is the mate of $\r$ with respect to this adjunction. Then $(T,\l)$ is a colax monad morphism.
\end{lem}
\begin{proof} We show $(T,\l)$ satisfies the pentagon and leave the triangle as an exercise. The pentagon for $(S,\rho)$ says that the left pasted squares 
\[\raisebox{35pt}{\xymatrix{\cdot \ar[r]^{HH} \ar@{}[dr]|{\mu \Searrow}& \cdot \\ \cdot \ar[u]^1\ar[r]_H \ar@{}[dr]|{\r \Searrow} & \cdot \ar[u]_1 \\ \cdot \ar[u]^S \ar[r]_K & \cdot \ar[u]_S}} =  \raisebox{35pt}{\xymatrix{ \cdot \ar[r]^H \ar@{}[dr]|{\r \Searrow} & \cdot \ar[r]^H \ar@{}[dr]|{\r \Searrow} & \cdot \\ \cdot \ar[u]^S \ar[r]_K \ar@{}[drr]|{\mu \Searrow} & \cdot \ar[r]_K \ar[u]^S & \cdot \ar[u]_S \\ \cdot \ar[u]^1 \ar[rr]_K & & \cdot \ar[u]_1}}\quad\quad\raisebox{35pt}{\xymatrix{\cdot \ar[r]^{HH} \ar@{}[dr]|{\mu \Swarrow} \ar[d]_1& \cdot \ar[d]^1 \\ \cdot \ar[d]_T \ar[r]_H \ar@{}[dr]|{\l \Swarrow} & \cdot \ar[d]^T \\ \cdot  \ar[r]_K & \cdot }} =  \raisebox{35pt}{\xymatrix{ \cdot \ar[r]^H  \ar[d]_T \ar@{}[dr]|{\l \Swarrow} & \cdot \ar[d]_T \ar[r]^H \ar@{}[dr]|{\l \Swarrow} & \cdot \ar[d]^T \\ \cdot \ar[d]_1 \ar[r]_K \ar@{}[drr]|{\mu \Swarrow} & \cdot \ar[r]_K  & \cdot \ar[d]^1 \\ \cdot \ar[rr]_K & & \cdot }}
\]
are equal in $\R${\bf adj}. By Theorem \ref{dblmatesthm} the pasted composites of their mates in $\L${\bf adj}, displayed on the right above, also agree.
\end{proof}

Of course, analogous results hold with any 2-category in place of $\CAT$; Theorem \ref{dblmatesthm} asserts that the functors $\L${\bf Adj}$,\R${\bf Adj}$\colon$2-$\CAT \rightrightarrows {\bf DblCAT}$ are isomorphic.

\subsection{Parameterized mates}\label{paramatesssec}

By a lemma below, in the context of a two-variable adjunction, or more generally a parameterized adjunction, the mates correspondences for the adjunctions obtained by fixing the parameter are natural in the parameter. This means that the two sets of mates assemble into natural transformations of two variables. We say that natural transformations corresponding in this way are \emph{parameterized mates}. We are not aware if this correspondence has been studied before, but it is essential to describe the interactions between awfs and two-variable adjunctions. The following lemmas establish the bare bones of this theory. 

First, we prove that if we fix one of the variables in a natural transformation between bifunctors which are pointwise adjoints and then take mates, the resulting \emph{pointwise mates} assemble to give a natural transformation between the appropriate bifunctors.

\begin{lem}\label{natmateslem} Suppose given a pair of left-closed bifunctors $\otimes, \otimes'$; ordinary functors $K, M, N$; and a natural transformation $\l_{k,m} \colon Kk \otimes' Mm \to N(k \otimes m)$ as displayed
\[\xymatrix{ \K \times \M \ar[d]_{\otimes} \ar[r]^-{K \times M} \ar@{}[dr]|{\l\, \Swarrow} & \K' \times \M' \ar[d]^{\otimes'} \\ \N \ar[r]_N & \N}\] Let $\r_{k,-}$ denote the mate of the natural transformation $\l_{k,-}$ with respect to the adjunctions $k \otimes - \dashv \hom(k,-)$ and $Kk \otimes' - \dashv \hom'(Kk,-)$. Then the $\r_{k,-}$ are also natural in $\K$ and assemble into a natural transformation $\r_{k,n} \colon M\hom(k,n) \to \hom'(Kk, Nn)$ \[\xymatrix{ \M \ar[r]^M \ar@{}[dr]|{\Searrow\, \rho} & \M' \\
\K^\op \times \N \ar[u]^{\hom} \ar[r]_-{K \times N} & \K'^\op \times \N' \ar[u]_{\hom'}}\]
\end{lem}
\begin{proof} Naturality of $\l$ in $\K$ says that for any $f \colon k' \to k$ in $\K$, the pasted composites 
\[\xymatrix@C=50pt{\M \ar[d]_{k \otimes-} \ar[r]^1 \ar@{}[dr]|{{}^{f\otimes-}\Swarrow}& \M  \ar[d]|{k'\otimes -} \ar@{}[dr]|{{}^{\l_{k',-}}\Swarrow} \ar[r]^M & \M' \ar[d]|{Kk' \otimes'-} \\ \N \ar[r]_1   & \N \ar[r]_N  & \N'} = \xymatrix@C=50pt{\M \ar[d]|{k'\otimes-} \ar[r]^M \ar@{}[dr]|{{}^{\l_{k,-}}\Swarrow} & \M' \ar[d]|{Kk \otimes'-} \ar[r]^1 \ar@{}[dr]|{{}^{Kf\otimes'-}\Swarrow}& \M' \ar[d]|{Kk'\otimes'-} \\ \N \ar[r]_N  & \N' \ar[r]_1 & \N' }\]
are equal. By Theorem \ref{dblmatesthm}, the pasted composites
\[\xymatrix@C=45pt{\M \ar[r]^1 \ar@{}[dr]|{\Searrow^{\hom(f,-)}}& \M  \ar@{}[dr]|{\Searrow^{\r_{k',-}}} \ar[r]^M & \M'  \\ \N \ar[r]_1  \ar[u]|(.35){\hom(k,-)} & \N \ar[r]_N \ar[u]|(.3){\hom(k',-)} & \N'\ar[u]|(.35){\hom'(Kk',-)}} = \xymatrix@C=45pt{\M \ar[r]^M \ar@{}[dr]|{\Searrow^{\r_{k,-}}} & \M' \ar[r]^1 \ar@{}[dr]|{\Searrow^{\hom'(Kf,-)}}& \M' \\ \N \ar[r]_N  \ar[u]|(.35){\hom(k,-)}& \N' \ar[r]_1 \ar[u]|(.3){\hom'(Kk,-)} & \N' \ar[u]|(.35){\hom'(Kk',-)} }\] are also equal, which says that the $\r_k$ are natural in $\K$.
\end{proof}

The following lemma establishes the \emph{parameterized mates correspondence}. 

\begin{lem}\label{paramateslem} Given two-variable adjunctions $(\otimes,\homl,\homr)$, $(\otimes',\homl',\homr')$ and functors $K, M, N$ as below, there is a bijective correspondence between natural transformations 
\[\xymatrix@C=20pt{ \K \times \M \ar[d]_{\otimes} \ar[r]^-{K \times M} \ar@{}[dr]|{\l\,\Swarrow} & \K' \times \M' \ar[d]^{\otimes'} \\ \N \ar[r]_N & \N} \xymatrix@C=20pt{ \M \ar[r]^M \ar@{}[dr]|{\Searrow\, \rho^\ell} & \M' \\
\K^\op \times \N \ar[u]^{\homl} \ar[r]_-{K^\op \times N} & \K'^\op \times \N' \ar[u]_{\homl'}} \xymatrix@C=20pt{ \K \ar[r]^K \ar@{}[dr]|{\Searrow\, \rho^r} & \K' \\
\M^\op \times \N \ar[u]^{\homr} \ar[r]_-{M^\op\times N} & \M'^\op \times \N' \ar[u]_{\homr'}}  \]
obtained by applying the pointwise mates correspondence to either variable. 
\end{lem}
\begin{proof}
By symmetry, it suffices to show that if we fix $\K$ and takes pointwise mates to define $\r^\ell$ from $\l$ and then fix $\N$ and take pointwise mates to define $\r^r$ from $\r^\ell$, the result is the same as fixing $\M$ and taking pointwise mates to define $\r^r$ from $\l$. This follows from the formulas (\ref{matesformula}), the compatible hom-set isomorphisms (\ref{enrichedhomset}) and a diagram chase. We leave the details as an exercise to the reader with the following hint: when in a sequence of composable arrows, one sees the unit followed by arrows in the image of the right adjoint, this asserts that the composite is adjunct to whatever remains when the unit and the right adjoint are erased. We made frequent use of this observation and its dual.
\end{proof}

A careful statement of the ``multi-functoriality'' of the parameterized mates correspondence, the appropriate analog of Theorem \ref{dblmatesthm}, involves category objects in the category of multicategories equipped with certain additional structure. This result will appear in a separate paper \cite{chenggurskiriehlparametrised}, joint work with Eugenia Cheng and Nick Gurski. For  present purposes, we only need a preliminary lemma in this direction.

\begin{lem}\label{paracomplem} Composition of parameterized mates in any of the three variables with ordinary mates pointing in compatible directions is well-defined.
\end{lem}
\begin{proof} Suppose $\l$, $\r^\ell$, $\r^r$ are parameterized mates as in Lemma \ref{paramateslem}, and suppose $\a$ and $\b$ are mates with respect to the top squares of the following diagram in $\L${\bf adj}$\cong \R${\bf adj}. 
\[\xymatrix@!0@R=45pt@C=70pt{\J \ar@<-.75ex>[d]_-T \ar@{}[d]|-{\dashv} \ar[r]^J & \J'\ar@<-.75ex>[d]_-{T'} \ar@{}[d]|-{\dashv} &\J \ar@<-.75ex>[d]_-T \ar@{}[d]|-{\dashv} \ar[r]^J & \J'\ar@<-.75ex>[d]_-{T'} \ar@{}[d]|-{\dashv}\\ \K \ar@<-.75ex>[u]_-S \ar[r]_K \ar@<-.75ex>[d]_(.4){-\otimes m} \ar@{}[d]|(.4){\dashv} & \K' \ar@<-.75ex>[u]_-{S'} \ar@<-.75ex>[d]_(.6){-\otimes' Mm} \ar@{}[d]|(.6){\dashv}  &\K \ar@<-.75ex>[u]_-S \ar[r]_K \ar@<-.75ex>[d]_(.4){\homl(-,n)} \ar@{}[d]|(.4){\dashv} & \K' \ar@<-.75ex>[u]_-{S'} \ar@<-.75ex>[d]_(.6){\homl'(-,Nn)} \ar@{}[d]|(.6){\dashv}\\ \N \ar[r]_{N} \ar@<-.75ex>[u]_(.6){\homr(m,-)} & \N' \ar@<-.75ex>[u]_(.4){\homr'(Mm,-)} & \M^\op \ar@<-.7ex>[u]_(.6){\homr(-,n)} \ar[r]_M & \M'^\op \ar@<-.75ex>[u]_(.4){\homr'(-,Nn)} }\]
Applying Theorem \ref{dblmatesthm} and Lemma \ref{natmateslem} to the left-hand rectangle, we see that 
\begin{align*} &\xymatrix{ T'J \otimes' M \ar[r]^-{\a \otimes' 1} & KT \otimes' M \ar[r]^-{\l_{T,1}} & N(T \otimes -)} \quad \quad \mathrm{and} \\ &\xymatrix{JS\homr(-,-) \ar[r]^-{\b_{\homr}} & S'K\homr(-,-) \ar[r]^-{S'\r^r} & S'\homr'(M,N)}\\ \intertext{are mates; from the right-hand rectangle, we see that this second natural transformation and} &\xymatrix@C=30pt{ M\homl(T,-) \ar[r]^-{\r^\ell_{T,1}} & \homl'(KT,N) \ar[r]^-{\homl'(\a,N)} & \homl'(T'J,N)}\end{align*} are mates. By Lemma \ref{paramateslem}, the three composite natural transformations are parameterized mates.
\end{proof}

As a consequence, algebraic Quillen two-variable adjunctions pointing in the direction of the left adjoints can be composed in any of their variables with algebraic Quillen adjunctions pointing also in the direction of the left adjoints; see Lemma \ref{complem}.

\section{Preliminaries on algebraic weak factorization systems}\label{prelimsec}

We briefly review a few key topics: lifting properties, weak factorization systems, functorial factorizations, algebraic weak factorization systems, and the algebraic small object argument.

\subsection{Weak factorization systems}\label{wfsssec}

We write $\1, \2, \3, \4$ for the categories assigned to these ordinals; e.g., $\2$ is the ``walking arrow'' category, $\3$ is the free category containing a composable pair of arrows, and so on. The functor category $\M^\2$ is the category whose objects are arrows in $\M$, depicted vertically, and whose morphisms $(u,v) \colon f \To g$ are commutative squares \begin{equation}\label{square}\xymatrix{ \cdot \ar[d]_f \ar[r]^u & \cdot \ar[d]^g \\ \cdot \ar[r]_v & \cdot}\end{equation} Any such square presents a \emph{lifting problem} of $f$ against $g$; a solution would be an arrow from the bottom left to the upper right such that both resulting triangles commute. If every lifting problem presented by a morphism $f \To g$ has a solution, we say that $f$ has the \emph{left lifting property} against $g$ and, equivalently, that $g$ has the \emph{right lifting property} against $f$.

\begin{defn}[I.2.3, I.2.4]\label{defn:wfs} A \emph{weak factorization system} $(\mathcal{L},\mathcal{R})$ on $\M$ consists of two classes of morphisms such that \begin{myindentpar}{5em} \begin{itemize}
\item[(factorization)] every arrow of $\M$ can be factored as an arrow of $\mathcal{L}$ followed by an arrow of $\mathcal{R}$ \item[(lifting)] every lifting problem (\ref{square}) with $f \in \mathcal{L}$ and $g \in \mathcal{R}$ has a solution \item[(closure)] every arrow with the left lifting property against every arrow in $\mathcal{R}$ is in $\mathcal{L}$ and every arrow with the right lifting property against every arrow of $\mathcal{L}$ is in $\mathcal{R}$. \end{itemize}\end{myindentpar} In the presence of the first two axioms, the third can be replaced by \begin{myindentpar}{5em}\begin{itemize} \item[(closure$'$)] the classes $\mathcal{L}$ and $\mathcal{R}$ are closed under retracts \end{itemize}\end{myindentpar} by the so-called ``retract argument'' familiar from the model category literature.
\end{defn}

Adopting standard notation \begin{align*} \mathcal{L}^\boxslash &= \{ g \in \M^\2 \mid g\ \mathrm{has~the~right~lifting~property~against~all}\ f \in \mathcal{L}\} \\ {}^{\boxslash}\mathcal{R} &=\{ f \in \M^\2 \mid f\ \mathrm{has~the~left~lifting~property~against~all}\ g \in \mathcal{R}\} \end{align*} the lifting and closure axioms combine to assert that $\mathcal{R}= \mathcal{L}^\boxslash$ and $\mathcal{L} = {}^{\boxslash}\mathcal{R}$. In particular, it is clear that either class determines the other. For any class of morphisms $\cR$, the class ${}^\boxslash\cR$ is closed under coproducts, pushouts, (transfinite) composition, retracts, and contains the isomorphisms: precisely the familiar closure properties for the cofibrations in a model category. 

We will now ``categorify'' the notation just introduced.

\begin{defn}[I.2.25]\label{Jboxslashdefn} If $\J \to \M^\2$ is some subcategory of arrows, write $\J^\boxslash$ for the category whose objects are pairs $(f, \phi_f)$, where $f \in \M^\2$ and $\phi_f$ is a \emph{lifting function} that specifies a solution \[\xymatrix{ \cdot \ar[d]_j \ar[r]^a & \cdot \ar[d]^f \\ \cdot \ar[r]_b \ar@{-->}[ur]|{\phi_f(j,a,b)} & \cdot}\] to any lifting problem against some $j \in \J$ in such a way that the specified lifts commute with morphisms in $\J$. A morphism $(f,\phi_f) \to (g,\phi_g)$ is a morphism $f \To g$ in $\M^\2$ commuting with the chosen lifts.
\end{defn}

When $\J$ is discrete, the set of objects in the image of the forgetful functor $\J^\boxslash \to \M^\2$ is precisely the set $\J^\boxslash$ defined above. The category ${}^\boxslash\J$ is defined dually.

A \emph{functorial factorization} on $\M$ is a section $\vec{E} \colon \M^\2 \to \M^\3$ of the ``composition'' functor $\M^\3 \to \M^\2$; $\vec{E}$ is often described by a pair of functors $L,R \colon \M^\2 \rightrightarrows \M^\2$ whose respective codomain and domain define a common functor $E \colon \M^\2 \to \M$, as depicted below \begin{equation}\label{funfact}  \raisebox{.25in}{\xymatrix{ \cdot \ar[r]^u \ar[d]_f & \cdot \ar[d]^g \\ \cdot \ar[r]_v & \cdot}} \quad \stackrel{\vec{E}}{\mapsto}\quad \raisebox{.52in}{\xymatrix{\cdot \ar[r]^u \ar[d]_{Lf} & \cdot \ar[d]^{Lg} \\ Ef \ar[r]^{E(u,v)} \ar[d]_{Rf} & Eg \ar[d]^{Rg} \\ \cdot \ar[r]_v & \cdot}} \quad \stackrel{\circ}{\mapsto} \quad \raisebox{.25in}{\xymatrix{ \cdot \ar[r]^u \ar[d]_f & \cdot \ar[d]^g \\ \cdot \ar[r]_v & \cdot}}\end{equation} 

\begin{ntn}Throughout, the vector notation is used to decorate functors and natural transformations on diagram  categories whose essential data is described by one component; e.g., $E$ contains all the data of the action of $\vec{E}$ on morphisms.
\end{ntn}

\subsection{Algebraic weak factorization systems}\label{awfsssec}

The endofunctors $L, R$ arising from a functorial factorization $\vec{E}$ are equipped with canonical natural transformations $\vec{\e} \colon L \To 1$, $\vec{\eta} \colon 1 \To R$ whose components  are rearrangements of the functorial factorization; cf.~\S I.2.3.  A functorial factorization gives rise to an algebraic weak factorization system when this data can be extended to a compatible comonad and monad.

\begin{defn}[{Grandis-Tholen \cite{gtnatural}, Garner \cite{garnerunderstanding}}] An \emph{algebraic weak factorization system} $(\L,\R)$ on a category $\M$ consists of a comonad $\L = (L,\vec{\e},\vec{\d})$ and a monad $\R = (R,\vec{\eta}, \vec{\mu})$ arising from a functorial factorization and such that $(\d,\mu) \colon LR \To RL$ is a distributive law. 
\end{defn}

The functorial factorization of an algebraic weak factorization system determines an \emph{underlying weak factorization system} whose left and right classes are the retract closures of the classes of maps admitting $\L$-coalgebra and $\R$-algebra structures respectively. The comultiplication for the comonad and multiplication for the monad ensure that left and right factors are themselves $\L$-coalgebras and $\R$-algebras. Equivalently, the left and right classes consist of those maps that admit solutions to the lifting problems 
 \begin{equation}\label{eq:pointedstr} \xymatrix{ \dom f \ar[d]_f \ar[r]^{Lf} & Ef \ar[d]^{Rf}  & & \dom g \ar@{=}[r] \ar[d]_{Lg} & \dom g \ar[d]^g \\ \cod f \ar@{=}[r]  \ar@{-->}[ur]_s & \cod f & & Eg \ar[r]_{Rg}   \ar@{-->}[ur]^t& \cod g}\end{equation}  respectively; lifts precisely define (co)algebra structures for the pointed endofunctors $L$ and $R$. The (co)algebra structures of \eqref{eq:pointedstr} can be used to define  a canonical solution to any lifting problem in such a way that the canonical solution to the lifting problems posed in \eqref{eq:pointedstr} are $s$ and $t$:  \begin{equation}\label{chosenlift}\xymatrix@C=30pt{ \cdot \ar[r]^u \ar[d]_{Lf} & \cdot \ar[d]^{Lg}  \\ \cdot \ar@{-->}[r]^{E(u,v)} \ar[d]_{Rf} & \cdot \ar[d]^{Rg} \ar@<1ex>@{-->}[u]^t \\ \cdot \ar@<-1ex>@{-->}[u]_s \ar[r]_v & \cdot }\end{equation} In this way, all $\L$-coalgebras lift against all $\R$-algebras. Furthermore, morphisms of (co)algebras preserve the chosen solutions to lifting problems, defining functors \begin{equation}\label{eq:liftfunctor} \xymatrix{ \Lalg \ar[r]^-{\mathrm{lift}} & {}^\boxslash\Ralg & \Ralg \ar[r]^-{\mathrm{lift}} & \Lalg^\boxslash}\end{equation}

We call maps admitting $\L$-coalgebra structures \emph{cellular} to distinguish them from mere pointed endofunctor coalgebras. This distinction is classical: Quillen's original notion of model category did not include the closure axiom of Definition \ref{defn:wfs}, presumably because he wanted the cofibrations in his model structure on spaces to be what we'd term the cellular cofibrations: the relative cell complexes. One of the most interesting features of this work is the power of the cellularity condition illustrated by Theorems \ref{cgadjthm} and \ref{existthm} below.

\subsection{The algebraic small object argument}

The following theorem enables the theory of algebraic model categories. Here a category ``permits the small object argument'' if it is \emph{locally bounded}, a set theoretical condition developed by Freyd and Kelly that includes locally presentable categories as well as many categories of topological spaces \cite{freydkellycategories,kellylack}.

\begin{thm}[Garner \cite{garnerunderstanding}]\label{garnsthm} Suppose $\M$ permits the small object argument and $\J$ is any small category of arrows of $\M$. Then $\M$ has an awfs $(\L,\R)$ so that  there is \begin{myindentpar}{3em} \begin{itemize} \item[\emph{(I.2.26)}] a functor $\J \to \Lalg$ over $\M^\2$ universal among morphisms of awfs \item[\emph{(I.2.27)}] an isomorphism of categories $\Ralg \cong \J^\boxslash$ over $\M^\2$ \end{itemize}\end{myindentpar}
\end{thm}

We make frequent use of both universal properties. Indeed, the universal property (I.2.26) of the \emph{unit functor} $\J \to \Lalg$ is even stronger than originally stated. It was first extended in \S I.6.4, reproduced as Theorem \ref{universalSOAthm} below, and will be extended further in Theorem \ref{extendedUPthm}. The isomorphism (I.2.27) factors as \[\xymatrix{ \Ralg \ar[r]^-{\mathrm{lift}} & \Lalg^\boxslash \ar[r]^-{\mathrm{res}} & \J^\boxslash},\] where the second component is the restriction along the unit functor.
An easy consequence of (I.2.27) is that the class of algebras for the monad of a cofibrantly generated awfs is retract closed; cf.~I.2.30. For this reason,  the adjective ``cellular'' refers only to left maps.

\begin{ex}\label{Iex} The set $\I$ of inclusions of spheres as the boundary of disks in each dimension generates an awfs on spaces whose left class consists of retracts of relative cell complexes and whose right class consists of Serre fibrations that are also weak homotopy equivalences. Objects of the category $\I^\boxslash$ are maps equipped with lifted contractions filling spheres in the total space which are contractible in the base. The set of inclusions of simplicial spheres into the standard simplices of each dimension generates a similar awfs on the category of simplicial sets.
\end{ex}

\section{Morphisms of algebraic weak factorization systems}\label{morawfssec}

In \S\ref{algQuillsec}, we employ several flavors of morphisms of algebraic weak factorization systems to define algebraic model categories and algebraic Quillen functors of one and two variables. The constituent morphisms preserve either the left or right classes of awfs on different categories while interacting with both.  In this section, we leave aside the model structure context and focus instead on the categorical underpinnings of the two-variable algebraic Quillen functors that will appear in \S\ref{algQuillsec}. To contextualize the new notions appearing in \S\ref{adjawfsssec} below,  we first review the single-variable morphisms introduced in \S I.6.

\subsection{Morphisms} 

To warm up, let us consider the simplest case:  morphisms between two algebraic weak factorization systems on the same category. We write $\vec{E} = (L,R)$ for a functorial factorization in the sense of \eqref{funfact}.

\begin{defn}  A \emph{morphism of functorial factorizations} $\vec{E}  \to \vec{E'}$ consists of a natural transformation $\xi \colon E \To E'$ so that \begin{equation}\label{eq:funmor}\xymatrix@C=30pt@R=30pt@!0{ & \dom \ar[dl]_L \ar[dr]^{L'} \\ E \ar[rr]^\xi \ar[dr]_R & & E' \ar[dl]^{R'} \\ & \cod}\end{equation} commutes.
\end{defn}

The conditions \eqref{eq:funmor} assert that $(1,\xi) \colon L \To L'$ and $(\xi,1) \colon R \To R'$ are morphisms of pointed endofunctors. A morphism of functorial factorizations defines a pair of functors $L\text{-}\mathrm{\bf coalg} \to L'\text{-}\mathrm{\bf coalg}$ and $R'\text{-}{\mathrm{\bf alg}} \to R\text{-}{\mathrm{\bf alg}}$ over $\M^\2$  by post- and pre-composing with the relevant component of $\xi$. In the context of a wfs $(\cL,\cR)$, a functorial factorization gives rise to the following algebraic characterizations of the left and right classes: $f \in \cL$ and $g \in \cR$ if and only if $f$ admits the structure of an $L$-coalgebra and $g$ admits the structure of an $R$-algebra, as displayed in \eqref{eq:pointedstr}. In particular, the existence of a map \eqref{eq:funmor} implies that $\cL \subset \cL'$ and $\cR' \subset \cR$. 
  
 \begin{defn}[Garner, I.2.14]\label{defn:morawfs} A \emph{morphism of awfs} $\xi \colon (\L,\R) \to (\L',\R')$ consists of a natural transformation $\xi \colon E \To E'$ so that any, and hence all, of the following equivalent conditions hold \begin{itemize}\item $(1,\xi) \colon L \To L'$ is a colax comonad morphism and $(\xi,1) \colon R \To R'$ is a lax monad morphism  \item $\xi$ determines functors $\Lalg \to \Lpalg$ and $\Rpalg \to \Ralg$  \item $\xi$ is a morphism of functorial factorizations satisfying pentagons
 \[{\xymatrix@C=20pt@R=20pt@!0{ & E  \ar[ddl]_{\d} \ar[rr]^{\xi} & & E' \ar[ddr]^{{\d}'} & &&&  & E'R \ar[drr]^{E'(\xi,1)} & & \\ & & & & && ER \ar[ddr]_{{\mu}} \ar[urr]^{\xi_R} & & & & E'R'\ar[ddl]^{{\mu}'} \\ EL \ar[drr]_{\xi_L} & & & & E'L' & & & & & \\ & & E'L \ar[urr]_{E'(1,\xi)} & & &&  & E\ar[rr]_{\xi} & & E' & } } 
\] 
  \end{itemize}
 \end{defn}

\subsection{Colax morphisms, lax morphisms, and adjunctions}\label{adjssec}

Weak factorization systems on different categories are compared by means of a functor that preserves either the left or right classes. Consider functorial factorizations $\vec{Q}=(C,F)$ on $\M$ and $\vec{E}=(L,R)$ on $\K$.  

\begin{defn}\label{defn:colaxfunfact} A \emph{colax morphism of functorial factorizations} $(T,\l) \colon \vec{Q} \to \vec{E}$ consists of a functor $T \colon \M \to \K$ and a natural transformation $\lambda \colon TQ \To ET$ so that the following triangles commute. A \emph{lax morphism of functorial factorizations} $(S,\rho) \colon \vec{E} \to \vec{Q}$ consists of a functor $S \colon \K \to \M$ and a natural transformation $\rho \colon QS \To SE$ so that the following triangles commute. \begin{equation}\label{eq:colaxfun}\xymatrix{ & \dom T \ar[dl]_{TC} \ar[dr]^{LT} \\ TQ \ar@{-->}[rr]^\lambda \ar[dr]_{TF} & & ET \ar[dl]^{RT} \\ & \cod T} \qquad \qquad \xymatrix{ & \dom S \ar[dl]_{CS} \ar[dr]^{SL} \\ QS \ar@{-->}[rr]^\rho \ar[dr]_{FS} & & SE \ar[dl]^{SR} \\ & \cod S}\end{equation} 
\end{defn}

A colax morphism of functorial factorizations \eqref{eq:colaxfun} determines a functor $C$-{\bf coalg}$\to L$-{\bf coalg} lifting $T$; a lax morphism determines a functor $R$-{\bf alg}$\to F$-{\bf alg} lifting $S$. Consequently, if $T$ is part of a colax morphism of wfs, then $T$ preserves the left class, and if $S$ is part of a lax morphism, then $S$ preserves the right class. 

\begin{defn}[I.6.4, I.6.6]\label{defn:colaxmor} A \emph{colax morphism of awfs} $(T,\l) \colon (\C,\F) \to (\L,\R)$ consists of a functor $T \colon \M \to \K$ and a natural transformation $\l \colon TQ \To ET$ so that $(1,\l) \colon TC \To LT$ is a colax comonad morphism and $(\l,1) \colon TF \To RT$ is a colax monad morphism. A \emph{lax morphism of awfs} $(S,\rho) \colon (\L,\R) \to (\C,\F)$ consists of a functor $S \colon \K \to \M$ and a natural transformation $\rho \colon QS \To SE$ so that $(1,\rho) \colon CS \To SL$ is a lax comonad morphism and $(\rho,1) \colon FS \To SR$ is a lax monad morphism. 
\end{defn} 

\begin{rmk}
Colax comonad morphisms $TQ \To ET$ are in bijection with functors $\Calg \to \Lalg$ lifting $T$; dually, lax monad morphisms $FS \To SR$ are in bijection with functors $\Ralg \to \Falg$ lifting $S$. The extra conditions in Definition \ref{defn:colaxmor} are equivalent to asking that these lifted functors in fact define double functors between the double categories introduced in \S\ref{dblcatssec} below.
\end{rmk}

Now suppose $ T \colon \M \rightleftarrows \K \colon S$ is an adjunction. It is well known that the left adjoint $T$ preserves the left classes of wfs on $\M$ and $\K$ if and only if the right adjoint $S$ preserves the right classes. In the presence of functorial factorizations, an algebraic manifestation of this fact is encoded in the following definition. 

\begin{defn}
An \emph{adjunction of functorial factorizations} $(T,S,\l,\rho) \colon \vec{Q} \to \vec{E}$  is given by a pair of natural transformations 
 $\lambda \colon TQ \To ET$ and $\rho \colon QS \To SE$ that are mates with respect to the adjunctions \[\xymatrix{\M^{\bf 2} \ar@<-1ex>[d]_-{T^\2} \ar@{}[d]|-{\dashv} \ar[r]^Q & \M \ar@<-1ex>[d]_-T \ar@{}[d]|-{\dashv} \\ \K^{\bf 2} \ar@<-1ex>[u]_-{S^\2} \ar[r]_E & \K \ar@<-1ex>[u]_-S}\]  so that  $(T,\l) \colon \vec{Q} \to \vec{E}$ is a colax morphism of functorial factorizations or equivalently such that $(S,\rho) \colon \vec{E} \to \vec{Q}$ is a lax morphism of functorial factorizations.
 \end{defn}

Let $\M$ and $\K$ have awfs $(\C,\F)$ and $(\L,\R)$, respectively. 

\begin{defn}[I.6.10-13]\label{defn:adjawfs} An \emph{adjunction of awfs} $(T,S,\l,\r) \colon (\C,\F) \to (\L,\R)$ consists of an adjoint pair of functors together with mates $\l$ and $\r$, as above, such that $(S,\r)$ is a lax morphism of awfs and $(T,\l)$ is a colax morphism of awfs.
\end{defn}

\begin{ex} A morphism of awfs in the sense of Definition \ref{defn:morawfs} is an adjunction of awfs, the adjunction in question being the identity. 
\end{ex}

\begin{ex} The geometric realization--total singular complex adjunction defines an adjunction of awfs between the awfs of Example \ref{Iex}. The coalgebra structure of a monomorphism of simplicial sets amounts to a factorization of the map into countably many ``attaching stages'' in which simplices in the codomain are attached via their boundaries to the domain. The lift of $|-| \colon \sSet \to \Top$ to a functor between categories of algebraic cofibrations sends this monomorphism to a relative cell complex equipped with a canonical cellular decomposition. Simultaneously, the lift of $S \colon \Top \to \sSet$ to a functor between algebraic trivial fibrations sends maps with chosen lifted contractions of spheres in the total space that become contractible in the base to maps of simplicial sets with chosen fillers for simplicial spheres in the domain whose image bounds a simplex in the codomain.
\end{ex}

\subsection{Bicolax morphisms, bilax morphisms, and two-variable adjunctions}\label{adjawfsssec}

In \S\ref{adjssec}, we made use of the fact that an adjunction $T \colon \M \rightleftarrows \K \colon S$ induces an adjunction $T^\2 \colon \M^\2 \rightleftarrows \K^\2 \colon S^\2$ on arrow categories. Similarly, a two-variable adjunction induces a two-variable adjunction on arrow categories, though the constituent bifunctors are no longer defined pointwise.

A \emph{two-variable adjunction}  $(\otimes, \homl, \homr) \colon \K \times \M \to \N$ consists of three bifunctors
 \begin{equation}\label{twovaradj1}\xymatrix@C=30pt{ \K \times \M \ar[r]^-{-\otimes-} & \N}  \quad \xymatrix@C=30pt{ \K^\op \times \N \ar[r]^-{\homl(-,-)} & \M}\quad \xymatrix@C=30pt{ \M^\op \times \N \ar[r]^-{ \homr(-,-)} &\K}\end{equation} together with hom-set isomorphisms \begin{equation}\label{twovarhomset}\N(k \otimes m, n)  \cong \M(m, \homl(k,n)) \cong \K(k,\homr(m,n))\end{equation} natural in all three variables. In particular, these form a \emph{parameterized adjunction}: fixing any one variable gives rise to families of adjunctions in the usual sense. 
 
When $\K$ and $\M$ have pullbacks and $\N$ has pushouts, there is an induced two-variable adjunction \begin{equation}\label{twovaradj}\xymatrix@C=28pt{ \K^\2 \times \M^\2 \ar[r]^-{-\hat{\otimes}-} & \N^\2}  \ \xymatrix@C=28pt{ (\K^\2)^\op \times \N^\2 \ar[r]^-{\hhoml(-,-)} & \M^\2}\ \xymatrix@C=28pt{ (\M^\2)^\op \times \N^\2 \ar[r]^-{ \hhomr(-,-)} &\K^\2}\end{equation} defined at $i \colon A \to B \in \K$, $j \colon J \to K \in \M$, and $f \colon X \to Y \in \N$ by 
 \[\xymatrix{ A \otimes J \ar[d]_{i \otimes J} \ar@{}[dr]|(.8){\ulcorner} \ar[r]^{A \otimes j} & A \otimes K \ar[d] \ar@/^/[ddr]^{i \otimes K} &  \save[]\homl(B,X)\restore \ar@{-->}[dr]|(.6){\hhoml(i,f)} \ar@/^/[drr]^{\homl(B,f)} \ar@/_/[ddr]_{\homl(i,X)} \\ B \otimes J \ar[r] \ar@/_/[drr]_{B \otimes j} & \cdot \ar@{-->}[dr]|(.4){i \hast j} &   & \cdot \ar[r] \ar[d] \ar@{}[dr]|(.2){\lrcorner} &\save[]\homl(B,Y)\restore \ar[d]^{\homl(i,Y)} \\ & & B \otimes K & \save[]\homl(A,X)\restore \ar[r]_{\homl(A,f)} & \save[]\homl(A,Y)\restore}\] 
 The bifunctor $\hast$ is referred to as the \emph{pushout-product}; we call $\hhoml$ and $\hhomr$, defined analogously, \emph{pullback-homs}. 

In order to algebraicize Quillen bifunctors, we will make use of bicolax morphisms that are covariant in both variables and bilax morphisms of mixed variance. To introduce these notions, consider bifunctors $-\otimes- \colon \K \times \M \to \N$ and $\hom(-,-) \colon \M^\op \times \N \to \K$, abbreviated using exponential notation.

\begin{defn}  Given functorial factorizations $\vec{Q'}$ on $\K$, $\vec{Q}$ on $\M$, and $\vec{E}$ on $\N$, 
a \emph{bicolax morphism of functorial factorizations} lifting $\otimes$ and a \emph{bilax morphism of functorial factorizations} lifting $\hom$ are given respectively by natural transformations $\lambda \colon Q' \otimes Q \To E$ and $\rho \colon Q'\hhom \To \hom(Q,E)$ satisfying the displayed conditions (abbreviating the codomain and domain functors from arrow categories to their base with ``c'' and ``d'').
\begin{equation}\label{eq:colax2var} \xymatrix@C=35pt{ d \otimes Q \displaystyle\po_{d \otimes d} Q'\! \otimes d \ar[d]_{C'\! \hast C} \ar[r]^-{ 1 \otimes F \sqcup F'\! \otimes 1} & d \otimes c \displaystyle\po_{d \otimes d} c \otimes d \ar[d]^{L\hast}   \\ Q'\! \otimes Q \ar[r]^{\l} \ar[d]_{F'\! \otimes F} & E\hast \ar[d]^{R\hast }  \\ c \otimes c \ar@{=}[r] & c\hast } \xymatrix@C=45pt{ d^c \ar@{=}[r] \ar[d]_{C'\!\hhom} & d^c \ar[d]^{\hom(F,L)} \\ Q'\hhom \ar[d]_{F'\!\hhom} \ar[r]^{\rho^r} & E^Q \ar[d]^{\hhom(C,R)} \\ c^c \displaystyle\pb_{c^d} d^d \ar[r]_-*+{\labelstyle\hom(F,1) \times \hom(1,L)}  & c^Q \displaystyle\pb_{c^d} E^d }
\end{equation}
\end{defn}

This data defines functors  \[ \xymatrix{C'\text{-}{\mathrm{\bf coalg}} \times C\text{-}{\mathrm{\bf coalg}}  \ar[r]^-{\hotimes} &  L\text{-}{\mathrm{\bf coalg}}}  \qquad  \xymatrix{C\text{-}{\mathrm{\bf coalg}}^\op \times R\text{-}{\mathrm{\bf alg}} \ar[r]^-{\hhom} & F'\text{-}{\mathrm{\bf alg}}}\] lifting the pushout-product and pullback-hom.
When the functorial factorizations are part of wfs, the underlying non-algebraic content is that the pushout-product of two maps in the left classes on $\K$ and $\M$ is in the left class on $\N$ and that the pullback-hom of a map in the left class on $\M$ and a map in the right class on $\N$ is in the right class on $\K$.

Suppose $\K$, $\M$, and $\N$ are equipped with awfs $(\C',\F')$, $(\C,\F)$, and $(\L,\R)$ respectively. 
A two-variable adjunction of awfs $(\otimes, \homl, \homr) \colon (\C',\F') \times (\C,\F) \to (\L,\R)$ is given by a \emph{bicolax morphism of awfs} over $\otimes$ or dually by a \emph{bilax morphism of awfs} over either $\homl$ or $\homr$.  A bicolax morphism of awfs is a bicolax morphism of functorial factorizations in which the natural transformation satisfies three additional conditions: a pentagon and two hexagons relating the comultiplication in one of the domain variables with the (co)multiplications in the other two variables. However, we find a different formulation more enlightening. The data of a two-variable adjunction of awfs consists of lifted functors \begin{equation}\label{eq:twovaradjlifts}  \xymatrix{\Cpalg \times \Calg \ar[r]^-{\hotimes} & \Lalg} \end{equation} \[ \xymatrix{ \Cpalg^\op \times \Ralg \ar[r]^-{\hhoml} & \Falg} \qquad \xymatrix{ \Calg^\op \times \Ralg \ar[r]^-{\hhomr} &  \Fpalg} \]
characterized by the natural transformations displayed below
\begin{equation}\label{liftedfun}{ \xymatrix@R=5pt@C=20pt{  & d \otimes Q \displaystyle\po_{d \otimes d} Q' \otimes d \ar[dd]_{C' \hast C} \ar[r]^{ 1 \otimes F \sqcup F' \otimes 1} & d \otimes c \displaystyle\po_{d \otimes d} c \otimes d \ar[dd]^{L (-\hast-)} \\ **[r]\save[]{\vec{\l}:=}\restore & &   \\ &  Q' \otimes Q \ar[r]^{\l} & E(-\hast-) \\   & Q\hhoml \ar[r]_-{\r^{\ell}} \ar[dd]_{F\hhoml} & \homl(Q', E) \ar[dd]^{\hhoml(C', R)} 
 \\  **[r]\save[]{\vec{\r^\ell}:=}\restore& & \\  & \homl(c,c) \displaystyle\pb_{\homl(d,c)} \homl(d,d) \ar[r]_*+{\labelstyle\homl(C',1) \times \homl(1,L)} & \homl(Q',c) \displaystyle\pb_{\homl(d,c)} \homl(d,E) \\  & Q'\hhomr \ar[dd]_{F'\hhomr} \ar[r]_-{\r^r} & \homr(Q,E) \ar[dd]^{\hhomr(C,R)}\\  **[r]\save[]{\vec{\r^r}:=}\restore& & \\  &  \homr(c,c) \displaystyle\pb_{\homr(d,c)} \homr(d,d) \ar[r]_*+{\labelstyle \homr(C,1) \times \homr(1,L)} & {\homr(Q,c) \displaystyle\pb_{\homr(d,c)} \homr(d,E)} }}\end{equation}
 which satisfy (co)unit and (co)multiplication conditions encoding compatibility with the (co)monads.
 
\begin{defn}\label{twovaradjdefn} A \emph{two-variable adjunction of awfs} $\otimes \colon (\C',\F') \times (\C,\F) \to (\L,\R)$ is a two-variable adjunction $(\otimes, \homl, \homr) \colon \K \times \M \to \N$ equipped with lifted functors \eqref{eq:twovaradjlifts} characterized by natural transformations $\vec{\l}$, $\vec{\r^\ell}$, and $\vec{\r^r}$ whose components $\l$, $\r^\ell$, $\r^r$ 
\[\xymatrix@C=12pt{ \K^\2 \times \M^\2 \ar[d]_{\hast} \ar[r]^-{Q' \times Q} \ar@{}[dr]|{\l\Swarrow} & \K \times \M \ar[d]^{\otimes} \\ \N^\2 \ar[r]_E & \N}\xymatrix@C=12pt{ \M^\2 \ar[r]^Q \ar@{}[dr]|{\Searrow \rho^\ell} & \M \\ (\K^\2)^\op \times \N^\2 \ar[u]^{\hhoml} \ar[r]_-{Q' \times E} & \K^\op \times \N \ar[u]_{\homl}} \xymatrix@C=12pt{ \K^\2 \ar[r]^{Q'} \ar@{}[dr]|{\Searrow \rho^r} & \K \\
(\M^\2)^\op \times \N^\2 \ar[u]^{\hhomr} \ar[r]_-{Q\times E} & \M^\op \times \N \ar[u]_{\homr}}  \] are parameterized mates.
\end{defn}

\begin{ntn}\label{terriblentn}
In \S\S\ref{compsec}-\ref{cellsec}, in which we prove our main technical theorems characterizing two-variable adjunctions of awfs, we adopt the following notation. We fix a two-variable adjunction $(\otimes,\homl,\homr) \colon \K \times \M \to \N$ and let $(\C',\F')$, $(\C,\F)$, and $(\L,\R)$ denote awfs on $\K$, $\M$, and $\N$, respectively, extending functorial factorizations $\vec{Q}'$, $\vec{Q}$, and $\vec{E}$. We write $i \colon A \to B$, $j \colon J \to K$, and $f \colon X \to Y$ for generic elements of $\K^\2$, $\M^\2$, and $\N^\2$ respectively. Whenever we assume further that $i$ has the structure of a $\C'$-coalgebra, $j$ has the structure of a $\C$-coalgebra, or $f$ has the structure of an $\R$-algebra, we make this explicit.
\end{ntn}

\section{The composition criterion}\label{compsec}

In practice it is often easier to define a lifted bifunctor between categories of (co)algebras for awfs than it is to write down the characterizing natural transformation. In this section, we will develop a composition criterion that allows us to recognize bilax and bicolax morphisms of awfs ``in the wild.'' This will be the key tool in proof of our main existence theorem for morphisms between cofibrantly generated awfs.  In later sections, we will see that this is essentially our only trick for recognizing two-variable adjunctions of awfs.

For motivation, we begin in \S\ref{dblcatssec} with a digression on double categorical aspects of awfs. Then in \S\ref{arisingssec} we consider certain single-variable adjunctions derived from two-variable adjunctions to introduce ideas and notation that will be used later. In \S\ref{compssec}, we finally present the composition criterion, proving it first for single-variable adjunctions and then extending it immediately to two-variable adjunctions.

\subsection{Double categorical aspects of algebraic weak factorization systems}\label{dblcatssec}

The idea behind the composition criterion is motivated by the following collection of ideas due to Richard Garner, which we now summarize. Here composition means ``vertical'' composition of algebras or coalgebras for the monad or comonad of an awfs; we'll see shortly that these assemble into a double category. 

The vertical composition law is most clearly illustrated in the cofibrantly generated case. The category $\J^\boxslash$ of Definition \ref{Jboxslashdefn} is the category of vertical arrows and squares for a double category, also denoted $\J^\boxslash$, whose objects are horizontal arrows are simply those of $\M$.  The vertical composition is defined as follows. If $(f,\phi_f), (g,\phi_g) \in \J^\boxslash$ with $\cod f = \dom g$, their composite $(gf, \phi_{g} \bullet \phi_f)$ is given by \begin{equation}\label{boxcomp} \phi_g \bullet \phi_f (j,a,b) := \phi_f(j, a, \phi_g(j, fa, b)) \quad \quad \vcenter{\xymatrix{\cdot \ar[dd]_j \ar[rr]^a & & \cdot \ar[d]^f \\ && \cdot \ar[d]^g \\ \cdot \ar@{-->}[urr]_{\phi_g(j,fa,b)} \ar[rr]_b \ar@{-->}[uurr]^(.6)*-{\labelstyle\phi_{f}(j,a,\phi_g)} & & \cdot}}\end{equation}
There is a forgetful double functor $\J^\boxslash \to \SSq(\M)$ which restricts to the usual forgetful functor on the categories of vertical arrows. 

For essentially the same reason, the category $\Lalg$ embeds as the vertical arrows and squares of a double category $\Coalg(\L)$: \[\xymatrix{ \Lalg \displaystyle\pb_{\M} \Lalg \ar[r]^-{\circ} &  \Lalg \ar@<1ex>[r]^-{\dom} \ar@<-1ex>[r]_-{\cod} & \M \ar[l]|-{\id}}\] Objects are objects of $\M$, horizontal arrows are morphisms of $\M$, vertical arrows are $\L$-coalgebras, and squares are maps of $\L$-coalgebras. The essential point is that $\L$-coalgebras have a canonical composition law---the functor $\circ$ above---that is functorial against $\L$-coalgebra morphisms. This vertical composition is derived from the embedding (\ref{eq:liftfunctor}) and (\ref{boxcomp}): if $(i,s), (j,t) \in \Lalg$ with $\cod i = \dom j$, then the arrow $ji$ is canonically an $\L$-coalgebra with coalgebra structure $t\bullet s$ defined by \begin{equation}\label{coalgcomp}t\bullet s := \xymatrix{ \cod j \ar[r]^t & Ej \ar[rr]^-{E(E(1,j)\cdot s,1)} & & ER(ji) \ar[r]^{\mu_{ji}} & E(ji).}\end{equation} Dually to the construction of (\ref{boxcomp}), $t \bullet s$ is defined to be the canonical solution to the lifting problem displayed on the left
\[\xymatrix{ \dom j \ar@<-.5ex>[dd]_j \ar[r]^{E(1,j)\cdot s} & E(ji) \ar@<.5ex>[dd]^{R(ji)}  & & \dom i \ar@<-.5ex>[dd]_i \ar@{=}[r] & \dom i \ar[dd]_(.75){ji} \ar[r]^{L(ji)} & E(ji) \ar@<.5ex>[dd]^{R(ji)} \\  \ar@{-->}[r]_{E(E(1,j)s,1)}& \ar@<.5ex>@{-->}[u]^{\mu_{ji}} & & \ar@{-->}[r]^{E(1,j)} & \ar@{-->}[r]_{E(L(ji),1)} \ar@{-->}[ur]^1 & \ar@<.5ex>@{-->}[u]^(.4){\mu_{ji}}  \\ \cod j \ar@<-.5ex>@{-->}[u]_t \ar@{=}[r] & \cod j & & \cod i = \dom j \ar@<-.5ex>@{-->}[u]_s  \ar[r]_-j & \cod j \ar@{=}[r] &}\]
whose top component, by a monad triangle identity, is the canonical solution to the lifting problem displayed on the right.

There is an obvious forgetful double functor $\Coalg(\L) \to \SSq(\M)$ which factors through the left class of the underlying wfs of $(\L,\R)$. A double category $\Alg(\R)$ is defined similarly with composition law, dual to (\ref{coalgcomp}), defined by means of the vertical composition in $\Lalg^\boxslash$. While the definitions are fresh in mind, we prove a lemma that will be used later.

\begin{lem}\label{liftcomplem} For any awfs $(\L,\R)$, the functor $\xymatrix{\Ralg \ar[r]^-{\mathrm{lift}} & \Lalg^\boxslash}$ over $\M^\2$ preserves composition of algebras.
\end{lem}
\begin{proof} The functor ``lift'' assigns an $\R$-algebra $(g,t)$ the lifting function $\phi_{(g,t)}$ defined via \eqref{chosenlift}. Given composable $(f,s), (g,t) \in \Ralg$, we must show that $\phi_{(g,t)} \bullet \phi_{(f,s)}$, defined by the formula (\ref{boxcomp}), equals $\phi_{(gf,t\bullet s)}$. Using the dual to (\ref{coalgcomp}), the chosen solution $\phi_{(gf,t\bullet s)}((j,z),a,b)$ to a lifting problem (\ref{boxcomp}) against an $\L$-coalgebra $(j,z)$ is \[\xymatrix@C=18pt{**[r]{\cod j} \ar[r]^-z \ar[dr]_z & Ej \ar[r]^-{E(a,b)} \ar[dr]_(.45){\d_j} & E(gf) \ar[r]^{\d_{gf}} & EL(gf) \ar[rr]^{E(1, E(f,1))} && E(Lg \cdot f) \ar[r]^-{E(1,t)} & Ef \ar[r]^-s & **[l]{\dom f} \\ & Ez \ar[r]_{E(1,z)} & ELj \ar[ur]_{E(a,E(a,b))} }\] by naturality of $\d$ and the comultiplication compatibility condition for $z$. By definition $\phi_{(g,t)}\bullet \phi_{(f,s)} ((j,z),a,b)$ is \begin{equation*}\xymatrix{ \cod j \ar[r]^z & Ej \ar[r]^-{E(a, \phi_g)} & Ef \ar[r]^-s & \dom f}\end{equation*} where $\phi_g$ is shorthand for $\phi_{(g,t)}((j,z),fa,b) := t \cdot E(fa,b) \cdot z$. The lifting problem $(a,\phi_g) : j \To f$ factors as \[\xymatrix{ \cdot \ar@{=}[r] \ar[d]_j & \cdot \ar[r]^a \ar[d]^{Lj} & \cdot\ar[d]^{L(gf)} \ar@{=}[r] & \ar[d]^{Lg \cdot f} \cdot \ar@{=}[r] & \cdot \ar[d]^f \\ \cdot \ar[r]_z & \cdot \ar[r]_{E(a,b)} & \cdot \ar[r]_{E(f,1)} & \cdot \ar[r]_t & \cdot}\] Hence, $E(a, \phi_g)$ is the image of this factorization under $E \colon \M^\2 \to \M$, and therefore $\phi_{(gf,t\bullet s)} = \phi_{(g,t)} \bullet \phi_{(f,s)}$.
\end{proof}

These double categories capture the entire structure of the awfs $(\L,\R)$.

\begin{lem}[Garner, I.2.24]\label{awfscharlem} Either of the double categories $\Coalg(\L)$ or $\Alg(\R)$ completely determines the awfs $(\L,\R)$.
\end{lem}
\begin{proof} Given $\Alg(\R)$, the functorial factorization $\vec{E}$, and in particular the functor $L$ and counit $\vec{\e}$, can be read off from the unit $\vec{\eta}$ of the monad $\R$. The comultiplication $\d$ can be defined in terms of the algebra structure assigned to the composite of the free algebras $(Rf, \mu_f) \circ (RLf, \mu_{Lf})$ as follows: \begin{equation}\label{deltadefn}\xymatrix@C=30pt{ \d_f := Ef \ar[r]^-{E(L^2f,1)} & E(Rf \cdot RLf) \ar[r]^-{\mu_f \bullet \mu_{Lf}} & ELf}\end{equation}
See \S I.2.5 or \cite[4.10]{barthelriehl} for more details.
\end{proof}

The characterization of Lemma \ref{awfscharlem} extends to morphisms. Let $(\C,\F)$ and $(\L,\R)$ be awfs on $\M$ and $\K$ and let $T \colon \M \rightleftarrows \K \colon S$ be functors, not necessarily adjoint.
Lifted double functors $T \colon \Coalg(\C) \to \Coalg(\L)$ and $S \colon \Alg(\R) \to \Alg(\F)$ determine lifted functors $T \colon \Calg \to \Lalg$, $S \colon \Ralg \to \Falg$ by passing to the categories of vertical arrows.

\begin{lem}[Garner, I.6.9]\label{laxmorlem} 
A lifted double functor $S \colon \Alg(\R) \to \Alg(\F)$ is precisely a lifted functor $S \colon \Ralg \to \Falg$ that preserves the canonical composition of algebras, which is precisely a lax morphism of awfs $S \colon (\L,\R) \to (\C,\F)$. Dually, a lifted double functor $T \colon \Coalg(\C) \to \Coalg(\L)$ is precisely a composition-preserving lifted functor $T \colon \Calg \to \Lalg$, which is precisely a colax morphism of awfs $T \colon (\C,\F) \to (\L,\R)$.
\end{lem}
\begin{proof} A double functor $\Alg(\R) \to \Alg(\F)$ lifting $S$ is determined by a commuting diagram of functors \[\raisebox{47pt}{\xymatrix@C=30pt{ \Ralg \displaystyle\pb_{\K} \Ralg \ar[r]^-{\circ} \ar[d]_{\tilde{S} \times_S \tilde{S}} & \Ralg \ar@<1ex>[r]^-{\dom} \ar@<-1ex>[r]_-{\cod} \ar[d]_{\tilde{S}} & \K \ar[l]|-{\id} \ar[d]^S \\ \Falg \displaystyle\pb_{\M} \Falg \ar[r]^-{\circ}  & \Falg \ar@<1ex>[r]^-{\dom} \ar@<-1ex>[r]_-{\cod} & \M \ar[l]|-{\id}}} \]
A lifted functor, $\Ralg \to \Falg$ preserves composition of certain free algebras if and only if the characterizing natural transformation satisfies a pentagon involving the comultiplication. See I.6.9 for more details.
\end{proof}

\subsection{Adjunctions arising from two-variable adjunctions}\label{arisingssec}

Recall the notational conventions introduced in \ref{terriblentn}. We consider adjunctions \begin{equation}\label{iadj} \xymatrix{i \hast - \colon \M^\2 \ar@<1ex>[r] \ar@{}[r]|-{\perp} & \ar@<1ex>[l] \N^\2 \colon \hhoml(i,-)}\end{equation} obtained by fixing $i \colon A \to B \in \K$ in the induced two-variable adjunction (\ref{twovaradj}). Because the right closure $\homr$ won't appear in this section, we abbreviate $\hhoml$ to $\hhom$ and use exponential notation for $\homl$.
 We want to extend the definitions of \S\ref{adjssec} to include functors of the form \eqref{iadj}.
 
 By Lemma \ref{appellem} and its dual, lifts of $i \hast-$ and $\hhom(i,-)$ to functors on coalgebras and algebras correspond to natural transformations \[ \xymatrix{i \hast C \ar@{=>}[r]^-{\vec{\l(i)}} & L(i\hast-)} \quad \mathrm{and} \quad \xymatrix{F\hhom(i,-) \ar@{=>}[r]^-{\vec{\r(i)}} & \hhom(i,R),}\] that satisfy (co)unit and (co)multiplication conditions. The (co)unit condition defines the domain of $\vec{\l(i)}$ and the codomain of $\vec{\r(i)}$ in such a way that the (co)mul\-tip\-li\-ca\-tion condition for that component is automatic. Write $\l(i) =\cod \vec{\l(i)}$ and $\r(i) = \dom \vec{\r(i)}$ for the non-trivial components.  The (co)multiplication conditions for $\l(i)$ and $\r(i)$ are pentagons which appear in the statement of Lemma \ref{adjawfsredefn}. The (co)unit conditions are expressed by saying that $\l(i)$ and $\r(i)$ define \emph{colax} and \emph{lax morphisms of functorial factorizations}:
\begin{equation}\label{laxgammarho}\xymatrix@R=10pt@C=10pt{ {A \otimes Qj \displaystyle\po_{A \otimes J} B \otimes J} \ar[dd]_{i \hast Cj} \ar[rr]^-{A \otimes Fj \sqcup 1} & & A \otimes K \displaystyle\po_{A \otimes J} B \otimes J \ar[dd]^{L(i \hast j)} &  X^B \ar@{=}[rr] \ar[dd]_{C\hhom(i,f)} & & X^B \ar[dd]^{\hhom(i,Lf)} \\ \\ B \otimes Qj \ar[rr]^{\l(i)_j} \ar[dd]_{i \otimes Fj} & & E(i\hast j) \ar[dd]^{R(i \hotimes j)}  &  Q\hhom(i,f) \ar[dd]_{F \hhom(i,f)} \ar[rr]^{\rho(i)_f} & & Ef^B \ar[dd]^{\hhom(i,Rf)} \\ \\ B \otimes K \ar@{=}[rr] & & B \otimes L &  Y^B \displaystyle\pb_{Y^A} X^A \ar[rr]_-{1 \times Lf^A } & & Y^B \times_{Y^A} Ef^A}\end{equation}

\begin{defn}\label{newadjawfsdefn} The functors $i \hast -$ and $\hhom(i,-)$ form an \emph{adjunction of awfs} $(\C,\F) \to (\L,\R)$ if there exist mates $\l(i)$ and $\r(i)$ with respect to the adjunctions 
\begin{equation}\label{eq:imates}\xymatrix@C=50pt@R=30pt{ \M^\2 \ar[r]^Q \ar@<-1ex>[d]_{i\hast -} \ar@{}[d]|{\dashv} & \M \ar@<-1ex>[d]_{B\otimes -} \ar@{}[d]|{\dashv} \\ \N^\2 \ar@<-1ex>[u]_{\hhom(i,-)} \ar[r]_E & \N \ar@<-1ex>[u]_{(-)^B}}\end{equation}
such that $\vec{\l(i)}$ and $\vec{\r(i)}$ determine lifts of $i \hast -$ and $\hhom(i,-)$.
\end{defn} 

In analogy with Definition \ref{defn:adjawfs}, any such adjunction of awfs is determined by the pair $(i \hast -, \l(i))$ or the pair $(\hhom(i,-), \r(i))$ alone. This is the result of Lemma \ref{adjawfsredefn} below. Its proof, via Theorem \ref{dblmatesthm}, requires some preparation. Unlike adjunctions of the form $T^\2 \colon \M^\2 \rightleftarrows \N^\2 \colon S^\2$ defined pointwise by some adjunction $T \colon \M \rightleftarrows \N \colon S$, the functors $i \hast -$ and $\hhom(i,-)$ do not preserve composability of vertical arrows. Instead, they induce an adjunction between the categories of composable triples of arrows in $\M$ and in $\N$. Writing $\iota$ and $\pi$ for the obvious legs of the pushout and pullback cones, the diagram
\[ \xymatrix@R=10pt@C=15pt{ A \otimes J \displaystyle\po_{A \otimes I} B \otimes I \ar[dd]_{i\hast j} \ar[rr]^-{a \sqcup b} & & X \ar[dd]^f & & I\ar[dd]_j  \ar[rr]^{b^\sharp} & & X^B \ar[dd]^{\hhom(i,f)} \\ \\ B \otimes J \ar[d]_{B \otimes k} \ar[rr]^c & & Y\ar[dd]^g & &  J \ar[dd]_k \ar@{-->}[drr]_{c^\sharp}\ar[rr]^-{c^\sharp \times a^\sharp} & & Y^B \times_{Y^A} X^A \ar[d]^{\pi} \\ B \otimes K \ar@{-->}[drr]^e \ar[d]_{\iota} & &  & \leftrightsquigarrow & & & Y^B \ar[d]^{g^B} \\ A \otimes L \displaystyle\po_{A \otimes K} B \otimes K \ar[rr]_-{d \sqcup e} \ar[dd]_{i \hast l} & & Z \ar[dd]^h & & K \ar[dd]_l \ar[rr]^{e^\sharp} & & Z^B \ar[dd]^{\hhom(i,h)} \\ \\ B \otimes L \ar[rr]^z & & W & & L \ar[rr]^-{z^\sharp \times d^\sharp} & & W^B \times_{W^A} Z^A}\]
exhibits the adjoint correspondence: the top and bottom squares correspond via $i \hast - \dashv \hhom(i,-)$ and the middle quadrangles correspond via $B \otimes - \dashv (-)^B$. Because our focus will be on the top and bottom squares, we denote this adjunction by \[ (i \otimes -, i \otimes -) \colon \M^\4 \rightleftarrows \N^\4 \colon (\hhom(i,-),\hhom(i,-))\] 

We give another presentation of the mates correspondence of \eqref{eq:imates} that captures the full data of $\vec{\l(i)}$ and $\vec{\r(i)}$.  Adopting simplicial notation, write $s_1$ for precomposition with the functor $\4 \to \3$ that collapses the middle two objects of $\4$ to the middle object of $\3$. We consider mates with respect to the adjunctions 
\begin{equation}\label{biggammarho}\xymatrix@C=40pt@R=30pt{ \M^\2 \ar[r]^{\vec{Q}} \ar@<-1ex>[d]_{i \hast -} \ar@{}[d]|{\dashv} & \M^\3 \ar[r]^{s_1} & \M^\4 \ar@<-1ex>[d]_{(i\hast-,i\hast-)} \ar@{}[d]|{\dashv} \\ \N^\2 \ar[r]_{\vec{E}} \ar@<-1ex>[u]_{\hhom(i,-)} & \N^\3 \ar[r]_{s_1} & \N^\4 \ar@<-1ex>[u]_{(\hhom(i,-),\hhom(i,-))}}\end{equation} whose components at $j \colon J \to K \in \M$ and $f \colon X \to Y \in \N$ are
\begin{equation}\label{gammarhodefn}\small{\xymatrix@R=10pt@C=10pt{ {A \otimes Qj \displaystyle\po_{A \otimes J} B \otimes J} \ar[dd]_{i \hast Cj} \ar[rr]^-{A \otimes Fj \sqcup 1} & & A \otimes K \displaystyle\po_{A \otimes J} B \otimes J \ar[dd]^{L(i \hast j)} &  X^B \ar@{=}[rr] \ar[dd]_{C\hhom(i,f)} & & X^B \ar[dd]^{\hhom(i,Lf)} \\ \\ B \otimes Qj \ar[rr]^{\l(i)_j} \ar[dd]_{\iota} & & E(i\hast j) \ar@{=}[dd] &  Q\hhom(i,f) \ar@{=}[dd] \ar[rr]^-{\rho'(i)_f} & & Ef^B \times_{Ef^A} X^A \ar[dd]^{\pi} \\ & &  & & & \\ A \otimes K \displaystyle\po_{A \otimes Qj} B \otimes Qj \ar[dd]_{i \hast Fj} \ar[rr]^-{\l'(i)_j} & & E (i \hast j) \ar[dd]^{R(i \hast j)} &  Q\hhom(i,f) \ar[dd]_{F \hhom(i,f)} \ar[rr]^{\rho(i)_f} & & Ef^B \ar[dd]^{\hhom(i,Rf)} \\ \\ B \otimes K \ar@{=}[rr] & & B \otimes K &  Y^B \displaystyle\pb_{Y^A} X^A \ar[rr]_-{1 \times Lf^A } & & Y^B \times_{Y^A} Ef^A}}\end{equation} We only consider pairs in which the left-hand mate defines a colax morphism of functorial factorizations and the right-hand mate defines a lax morphism of functorial factorizations. This requirement implies that the top-left and lower-right horizontal arrows have the form displayed in \eqref{gammarhodefn} and also imposes conditions on $\l'(i)$ and $\r'(i)$.

Obviously $\l'(i)$ determines $\l(i)$; under the hypothesis that the left-hand side is the mate of a lax morphism of functorial factorizations, the converse also holds. One leg of the cone defining $\l'(i)$ is $\l(i)$ and the other is necessarily a leg of the cone defining $L(i \hast -)$ on account of the appearance of the functor $L$ in the bottom component of the right-hand natural transformation. Similarly, when the mate of the right-hand side is a colax morphism of functorial factorizations, $\r(i)$ determines $\r'(i)$; its other leg is  a leg of the cone defining $F\hhom(i,-)$.

Extending the notation introduced above, write $\vec{\l(i)}$, $\vec{\l'(i)}$, $\vec{\rho'(i)}$, and $\vec{\rho(i)}$ for the natural transformations of the upper left, lower left, upper right, and lower right squares of (\ref{gammarhodefn}), respectively. Note $\vec{\l(i)}$ and $\vec{\rho'(i)}$ are mates and $\vec{\l'(i)}$ and $\vec{\rho(i)}$ are mates with respect to 
\[\xymatrix@C=50pt@R=15pt{ \M^\2 \ar[r]^{C} \ar@<-1ex>[dd]_{i \hast -} \ar@{}[dd]|{\dashv} & \M^\2 \ar@<-1ex>[dd]_{i\hast-} \ar@{}[dd]|{\dashv} & & \M^\2 \ar@<-1ex>[dd]_{i \hast -} \ar@{}[dd]|{\dashv} \ar[r]^F & \M^\2 \ar@<-1ex>[dd]_{i \hast -} \ar@{}[dd]|{\dashv}\\ & & \text{and} & &  \\ \N^\2 \ar[r]_{L} \ar@<-1ex>[uu]_{\hhom(i,-)} & \N^\2 \ar@<-1ex>[uu]_{\hhom(i,-)} & & \N^\2 \ar@<-1ex>[uu]_{\hhom(i,-)} \ar[r]_R & \N^\2 \ar@<-1ex>[uu]_{\hhom(i,-)}}\] respectively. Indeed $\l(i)$ and $\r(i)$ are mates if and only if \eqref{gammarhodefn} are mates, by Theorem \ref{dblmatesthm} and a diagram chase left to the reader.

With (\ref{gammarhodefn}), we can now be more explicit about the conditions on the natural transformations that satisfy Definition \ref{newadjawfsdefn}.

\begin{lem}\label{adjawfsredefn} An adjunction of awfs $(i \hast -, \hom(i,-))\colon (\C,\F) \to (\L,\R)$ is determined by either 
\begin{itemize} 
\item a \emph{colax morphism of awfs} $(i \hotimes -,\lambda(i)) \colon (\C,\F) \to (\L,\R)$, i.e., a natural transformation $\l(i)$ satisfying \emph{(\ref{laxgammarho})} and the pentagons
\[\small{\xymatrix@C=32pt@R=20pt@!0{ & i \hast C \ar[ddl]_{i \hast \vec{\d}} \ar[rr]^{\vec{\l(i)}} & & L(i \hast -) \ar[ddr]^{\vec{\d}_{i\hast-}} & &&  & R(i \hast F) \ar[drr]^{R\vec{\l'(i)}} & & \\ & & & & & **[r]i \hast F^2 \ar[ddr]_{i\hast \vec{\mu}} \ar[urr]^{\vec{\l'(i)_F}} & & & & **[l]R^2(i \hast -) \ar[ddl]^{\vec{\mu}_{i\hast-}} \\ **[r]i \hast C^2 \ar[drr]_{\vec{\l(i)_C}} & & & & **[l]L^2(i \hast -) & & & & & \\ & & L(i \hast C) \ar[urr]_{L\vec{\l(i)}} & & & & i \hast F \ar[rr]_{\vec{\l'(i)}} & & R(i \hast -) & } } 
\] 
\item a \emph{lax morphism of awfs} $(\hhom(i,-),\rho(i)) \colon (\L,\R) \to (\C,F)$, i.e., a natural transformation $\r(i)$ satisfying \emph{(\ref{laxgammarho})} and the pentagons
\[\small{\xymatrix@C=28pt@R=20pt@!0{ &  \save[]C\hhom(i,-) \ar[ddl]_{\vec{\d}_{\hhom(i,-)}} \ar[rr]^{\vec{\r'(i)}}\restore & & \save[] \hhom(i,L) \ar[ddr]^{\hhom(i,\vec{\d})}\restore &   \\ & & &  &  \\ **[r]\save[]C^2\hhom(i,-) \ar[drr]_{C\vec{\r'(i)}}\restore & & & & **[l]\save[]\hhom(i,L^2)\restore  \\ & & \save[] C\hhom(i,L) \ar[urr]_{\vec{\r'(i)_L}} \restore& & }\quad \xymatrix@C=28pt@R=20pt@!0{&  & F\hhom(i,R) \ar[drr]^{\vec{\rho(i)_R}} & & \\ **[r]F^2\hhom(i,-) \ar[ddr]_{\vec{\mu}_{\hhom(i,-)}} \ar[urr]^{F\vec{\rho(i)}} & & & & **[l]\hhom(i,R^2) \ar[ddl]^{\hhom(i,\vec{\mu})} \\ & & & & \\  & F\hhom(i,-) \ar[rr]_{\vec{\rho(i)}} & & \hhom(i,R) &}} 
\] 
\end{itemize} 
\end{lem}
\begin{proof}
When $\l(i)$ and $\r(i)$ are mates, so are $\vec{\l(i)}$ and $\vec{\r'(i)}$ and $\vec{\l'(i)}$ and $\vec{\r(i)}$; hence, by Theorem \ref{dblmatesthm}, the top pentagon in each column commutes if and only if the bottom one does.
\end{proof}

\subsection{Composition criterion}\label{compssec}

The idea for our composition criterion begins with the following observation about the vertical composition law for $\J^\boxslash$ defined in \eqref{boxcomp}. The square $(f,1) : (gf, \phi_g \bullet \phi_f) \to (g, \phi_g)$ is a morphism in $\J^\boxslash$, but $(1,g) \colon f \To gf$ is not. However this latter map, appearing as the middle square below, does preserve solutions to some lifting problems: namely those, depicted in the left hand square below, whose bottom arrow is the solution specified by $\phi_g$ to the composite lifting problem  \begin{equation}\label{someliftspreserved}\xymatrix@=30pt{ \cdot \ar[d]_j \ar[r]^a & \cdot \ar[d]_(.4)f \ar@{=}[r] & \cdot \ar[d]^{gf} \ar[r]^f & \cdot \ar[d]^g \\ \cdot \ar@/_1pc/[rrr]_b \ar@{-->}[urrr]|{\phi_g(j,fa,b)} \ar@{-->}[r] & \cdot \ar[r]_g & \cdot \ar@{=}[r] & \cdot}\end{equation}

\begin{rmk}
Similarly, for any awfs $(\L,\R)$ and composable $\R$-algebras $f$ and $g$, the square $(f,1) \colon gf \To g$ is a morphism of $\R$-algebras, while $(1,g) \colon f \To gf$ only preserves the canonical solutions to lifting problems of the form (\ref{someliftspreserved}). 
\end{rmk}

Composing a lifted functor $\hhom(i,-) \colon \Ralg \to \Falg$ with the embedding (\ref{eq:liftfunctor}), $\R$-algebras are mapped to arrows which have chosen solutions to lifting problems against $\C$-coalgebras. If these chosen solutions satisfy the criterion of the following theorem, then the lifted functor determines an adjunction of awfs.

\begin{thm}[Composition criterion]\label{compcritthm} A lifted functor $\hhom(i,-)\colon \Ralg \to \Falg$ defines a lax morphism of awfs if and only if the lifting functions assigned to a composable pair $f,g \in \Ralg$ have the following property: given a lifting problem $(a,b\times c)$ between a $\C$-coalgebra $j$ and $\hhom(i,gf)$, composition with the right square of the rectangle below determines a lifting problem against $\hhom(i,g)$. The chosen solution $d$ determines a lifting problem against $\hhom(i,f)$ whose solution $e$ should be the chosen solution to the original lifting problem. 
\begin{equation}\label{crazycomp} \xymatrix@R=30pt{ J \ar[r]^a \ar[d]_j & X^B \ar[d]|(.4){\hhom(i,f)} \ar@{=}[r] & X^B \ar[d]|(.6){\hhom(i,gf)} \ar[r]^{f^B} & Y^B \ar[d]|{\hhom(i,g)} \\ K \ar@{-->}[ur]^e \ar@{-->}[urr]^(.7)e \ar@{-->}[urrr]_(.8)d \ar@{-->}[r]_(.4){d \times c} \ar@/_2pc/[rr]_{b \times c} &  Y^B \times_{Y^A} X^A \ar[r]_{g^B \times_{g^A} 1} & Z^B \times_{Z^A} X^A \ar[r]_{1 \times_1 f^A} & Z^B \times_{Z^A} Y^A}\end{equation}
\end{thm}
\begin{proof}
By Lemmas \ref{appellem} and \ref{adjawfsredefn}, the lifted functor $\Ralg \to \Falg$ determines a lax morphism of awfs if and only if its characterizing natural transformation $\r(i)$ is such that the left-hand pentagon 
\[{\xymatrix@C=38pt@R=20pt@!0{ & & & & & Q\hhom(i,f) \ar[rr]^-{\r(i)_f} \ar[ddl]_(.6){\d_{\hhom(i,f)}} & & Ef^B \ar[ddr]^(.6){\d^B_f} \\ \\ & & & & **[r]{QC\hhom(i,f) \ar[drr]_(.4)*+{\labelstyle Q(1,\r'(i)_f)}} & & & & ELf^B \\ & Q\hhom(i,f) \ar[rr]^{\r'(i)_f} \ar[ddl]_(.6){\d_{\hhom(i,f)}} & & Ef^B \times_{Ef^A} X^A \ar[ddr]^(.6){ \d_f^B \times_{\d_f^A} 1} & & & Q \hhom(i,Lf)  \ar[urr]_(.6){\r(i)_{Lf}}  \\ \\ **[r]{QC \hhom(i,f) \ar[drr]_(.4)*+{\labelstyle Q(1,\r'(i)_f)}} & & & &  **[l]{ELf^B \times_{ ELf^A} X^A}   \\ & & Q\hhom(i,Lf) \ar[urr]_{\r'(i)_{Lf}}  && & }} \]
commutes. Projecting to one leg of the pullbacks, the left pentagon implies the right one, but an easy diagram chase---the essential point being that the other leg of $\r'(i)_f$ is defined to be a leg of $F\hhom(i,f)$---shows that the right pentagon also implies the left. Thus, it suffices to prove that the lifted functor satisfies the composition criterion if and only if this right-hand pentagon commutes.

Suppose $\hhom(i,-) \colon \Ralg \to \Falg$ is a lax morphism of awfs and consider composable $\R$-algebras $(f,s)$ and $(g,t)$. The lifted functor assigns the image of their composite $(gf, t\bullet s)$ the $\F$-algebra structure 
\begin{equation}\label{onedefn} \xymatrix@R=5pt@C=16pt{Q \hhom(i,gf) \ar[r]_-{\r(i)_{gf}} & E(gf)^B \ar[r]_{(t \bullet s)^B} & X^B = \\ Q \hhom(i,gf) \ar[r]_-{\r(i)_{gf}} & E(gf)^B \ar[r]_-{\d_{gf}^B} & EL(gf)^B \ar[rr]_-{E(1,E(f,1))^B} & & E (Lg \cdot f)^B \ar[r]_-{E(1,t)^B} & Ef^B \ar[r]_-{s^B} & X^B}\end{equation}
As for any $\F$-algebra structure, this map is the canonical solution assigned to the lifting problem \[\xymatrix@C=35pt{ X^B \ar@{=}[r] \ar[d]_{C \hhom(i,gf)} & X^B \ar[d]^{\hhom(i,gf)} \\ Q \hhom(i,gf) \ar[r]_{F\hhom(i,gf)} & Z^B \times_{Z^A} X^A} \] The composition criterion says that (\ref{onedefn}) should be obtained in the following manner. First, solve the composite lifting problem \[ \xymatrix@C=30pt{ X^B \ar@{=}[r] \ar[dd]_{C\hhom(i,gf)} & X^B \ar[r]^{f^B} \ar[dd] & Y^B \ar[dd]^{\hhom(i,g)} \\ {~} \ar@{-->}[r]^{Q(1, F\hhom(i,gf))} & {~} \ar@{-->}[r]^{Q(f^B, 1 \times_1 f^A)} & {~}\ar@{-->}@<1ex>[u]^{t^B\cdot \r(i)_g} \\ Q \hhom(i,gf) \ar@{-->}[ur]_1 \ar[r]_{F\hhom(i,gf)} \ar@{-->}@<-1ex>[u]_{\d} & Z^B \times_{Z^A} X^A \ar[r]_{1 \times_1 f^A} & Z^B \times_{Z^A} Y^A}\] in the manner displayed using the awfs $(\C,\F)$ and the $\F$-algebra structure assigned to $\hhom(i,g)$. The first two arrows in this lift compose to the identity by a triangle identity for the comonad $\C$; hence, by naturality of $\r(i)$, the canonical solution to this lifting problem is 
\[\xymatrix{ Q \hhom(i,gf) \ar[r]^-{\r(i)_{gf}} & E(gf)^B \ar[r]^-{E(f,1)^B} & Eg^B \ar[r]^{t^B} & Y^B}.\]
This arrow defines the other leg of the lifting problem \[\xymatrix{ X^B \ar@{=}[r] \ar[d]_{C \hhom(i,gf)} & X^B \ar[d]^{\hhom(i,f)} \\ Q \hhom(i,gf) \ar[r] & Y^B \times_{Y^A} X^A} \] whose canonical solution must agree with the $\F$-algebra structure of $\hhom(i,gf)$. This lifting problem factors as \[\xymatrix{ X^B \ar@{=}[r] \ar[d]|{C \hhom(i,gf)} & X^B \ar[d]|{\hhom(i,L(gf))} \ar@{=}[r] & X^B \ar[d]|{\hhom(i, Lg\cdot f)} \ar@{=}[r] & X^B \ar[d]|{\hhom(i,f)} \\ Q \hhom(i,gf) \ar[r]_{\r'(i)_{gf}} & E(gf)^B \times_{E(gf)^A} X^A \ar[r]_{E(f,1)^B \times_{E(f,1)^A} 1} & Eg^B \times_{Eg^A} X^A \ar[r]_{t^B \times_{t^A} 1} & Z^B \times_{Z^A} X^A} \] so its canonical solution, by naturality of $\r(i)$, is the composite 
\[\xymatrix@C=30pt{ Q\hhom(i,gf) \ar[r]^-{\d_{\hhom(i,gf)}} & QC\hhom(i,gf) \ar[r]^-{Q(1,\r'(i)_{gf})} & Q\hhom(i, L(gf)) \ar[r]^-{\r(i)_{L(gf)}} & EL(gf)^B \cdots}\]\[\xymatrix@C=30pt{ \cdots \ar[r]^-{E(1, E(f,1))^B} & E(Lg \cdot f)^B \ar[r]^-{E(1,t)^B} & Ef^B \ar[r]^{s^B} & X^B}\] 
which agrees with (\ref{onedefn}) if the pentagon for $gf$ is satisfied.

Conversely, suppose the lifted functor satisfies the composition criterion. Consider the composable pair of free $\R$-algebras $\xymatrix{ELf \ar[r]^{RLf} & Ef \ar[r]^{Rf} & Y}$. Employing the definition (\ref{deltadefn}) of $\d$, the upper right composite of the pentagon is the top composite of 
\[\xymatrix@C=13pt{  & Ef^B \ar[r]^{E(L^2f,1)^B} & E(Rf \cdot RLf)^B \ar[dr]^{(\mu_f \cdot \mu_{Lf})^B} \\ Q\hhom(i,f) \ar[d]_{\d_{\hhom(i,f)}} \ar[ur]^{\r(i)_f} \ar[r]
 & **[r]{Q \hhom(i, Rf \cdot RLf) \ar[ur]_{\r(i)_{Rf \cdot RLf}} \ar[d]^{\d_{\hhom(i, Rf\cdot RLf)}}} & & **[l]{ELf^B} \\ 
QC\hhom(i, f) \ar[r] \ar[d]_{Q(1,\r'(i)_f)} & **[r]{ QC\hhom(i, Rf\cdot RLf) \ar[d]^{Q(1,\r'(i)_{Rf \cdot RLf})}} & & **[l]{ERLf^B} \ar[u]^{\mu_{Lf}^B} \\ Q\hhom(i,Lf) \ar[dr]_{\r(i)_{Lf}} \ar[r] & **[r]{Q\hhom(i, L(Rf \cdot RLf)) \ar[dr]^(.6){\r(i)_{L(Rf \cdot RLf)}}} & & **[l]{E(LRf \cdot RLf)^B \ar[u]^{E(1, \mu_f)^B}} \\ & ELf^B \ar[r]_-{E(L^2f, E(L^2f, 1))^B} & EL(Rf \cdot RLf)^B \ar[ur]_{E(1, E(RLf,1))^B}  }\] The squares commute by naturality of $\r(i)$, $\d$, $\r'(i)$, and $\r(i)$; the definitions of the unlabeled arrows can be deduced from this. The octagon is exactly the composition criterion, in the form just deduced. Hence,the outer decagon commutes. The composite of the last four arrows along the bottom right is $(-)^B$ applied to an identity \[\xymatrix@C=40pt{ EL(Rf \cdot RLf) \ar[r]^-{E(1, E(RLf,1))} & E(LRf\cdot RLf) \ar[r]^-{E(1,\mu_f)} & ERLf \ar[d]^{\mu_{Lf}} \\ ELf \ar[u]^{E(L^2f, E(L^2f,1))} \ar[r]^{E(L^2f,1)} \ar@/_1pc/[rr]_1 & ERLf \ar[u]^{E(1,E(Lf,1))} \ar[ur]_1 & ELf }\] using functoriality of $E$ and two applications of a monad triangle identity. So the exterior of our decagon is the desired pentagon.
\end{proof}

A dual theorem describes those lifts of $i \hast -$ which determine colax morphisms of awfs. The two-variable version is now within reach.

\begin{thm}[Composition criterion]\label{twovarcompthm} Suppose $(\otimes,\homl,\homr) \colon \K \times \M \to \N$ is a two-variable adjunction between categories equipped with awfs $(\C',\F')$, $(\C,\F)$, $(\L,\R)$. A single lifted functor \[ \Cpalg \times \Calg \to \Lalg, \quad \Cpalg^\op \times \Ralg \to \Falg,\]\[\mathrm{or} \quad \Calg^\op \times \Ralg \to \Fpalg\]
specifies a two-variable adjunction of awfs if and only if it satisfies the criterion of Theorem \ref{compcritthm} or its dual, as appropriate, in each variable.
\end{thm} 
\begin{proof}
This follows easily from Theorem \ref{compcritthm} and the calculus of parameterized mates. Without loss of generality, suppose given $\Cpalg \times \Calg \to \Lalg$. Evaluating at $i \in \Cpalg$ defines a lifted functor $i \hast - \colon \Calg \to \Lalg$ characterized by a natural transformation $\l(i) \colon \cod i \otimes Q- \To E(i \hast -)$. By Theorem \ref{compcritthm}, the mates $\r^\ell(i) \colon Q\hhoml(i,-) \To \homl(\cod i, E-)$ specify lifted functors $\hhoml(i,-) \colon \Ralg \to \Falg$. Because the $\l(i)$ are natural in $\Cpalg$, so are the $\r^\ell(i)$ by Lemma \ref{natmateslem}. Using the definition of the lifted functor $\hhoml(i,-)$ in terms of the $\r^\ell(i)$ and an easy diagram chase, these assemble into a lifted bifunctor \[\hhoml(-,-)\colon \Cpalg^\op \times \Ralg \to \Falg.\] 

It remains only to show that the characterizing natural transformation of this functor is a parameterized mate of the characterizing natural transformation of the original $\Cpalg \times \Calg \to \Lalg$. By definition $\l \colon Q' \otimes Q \To E\hast $ and $\r^\ell \colon Q\hhoml \To \homl(Q',E)$ are obtained by composing $\l(C'-)$ and $\r^\ell(C'-)$ with the comonad counit. Explicitly, $\l_{i,-}$ and $\r^\ell_{i,-}$ are the pasted composites displayed below in the double categories $\L${\bf Adj} and $\R${\bf Adj} respectively
\[\xymatrix@R=30pt@C=90pt{ \M^\2 \ar@<-1ex>[d]_(.3){i \hast -} \ar@{}[d]|{\dashv} \ar[r]^1 \ar@{}[dr]|(.3){\Searrow(\vec{e}_i)^*}|(.7){(\vec{\e}_i)_*\Swarrow} & \M^\2 \ar@<-1ex>[d]_(.3){C'i \hast -} \ar@{}[d]|{\dashv}  \ar[r]^Q \ar@{}[dr]|(.3){\Searrow \r^\ell(C'i)}|(.7){\l(C'i)\Swarrow} & \M \ar@<-1ex>[d]_(.3){Q'i \otimes -} \ar@{}[d]|{\dashv} \\ \N^\2 \ar[r]_1 \ar@<-1ex>[u]_(.3){\hhoml(i,-)} & \N^\2 \ar@<-1ex>[u]_(.3){\hhoml(C'i,-)} \ar[r]_E & \N \ar@<-1ex>[u]_(.3){\homl(Q'i,-)}}\] Hence, $\l_{i,-}$ and $\r^\ell_{i,-}$ are pointwise mates by Theorem \ref{dblmatesthm} and thus parameterized mates by Lemma \ref{paramateslem}.
\end{proof}

\section{The cellularity and uniqueness theorems}\label{cellsec}

With Theorem \ref{twovarcompthm}, we know how to recognize two-variable adjunctions of awfs should we happen to stumble upon one. In this section we will prove a powerful existence theorem that enables us to construct these structures explicitly provided the domain awfs are cofibrantly generated and the generators satisfy a simple cellularity condition. 

For context, we begin in \S\ref{oldcellssec} by reviewing previous results in this direction for ordinary adjunctions of awfs. In \S\ref{twovarcellssec}, we use the composition criterion to give first a mild and then a dramatic extension of previous results, proving the much-advertised cellularity theorem. We would like to conclude further that such extensions are unique; this ends up being surprisingly difficult.

With this aim in mind, in \S\ref{upssec} we further extend the universal property of the unit functor constructed via Garner's small object argument, proving that the previous universality among adjunctions of awfs still holds with the extended terminology of \S\ref{arisingssec}. The general structure of the proof parallels our original argument, though the technical details are somewhat more complicated. The desired uniqueness theorem is now an immediate corollary.

\subsection{Cellularity and adjunctions of awfs}\label{oldcellssec}

We lay the groundwork for our cellularity theorem by reviewing previous results in this direction:  Lemma \ref{laxmorlem} encodes a composition criterion that can be used to characterize the adjunctions of awfs whose domain is cofibrantly generated. The non-trivial direction of the following theorem was first suggested by Mike Shulman; his proof appears as I.6.17. Below, we give a streamlined argument, whose essential details are the same but whose presentation is more conceptual.

\begin{thm}[I.6.17]\label{cgadjthm} Suppose $\M$ has an awfs $(\C,\F)$ generated by $\J$ and $\K$ has an awfs $(\L,\R)$, not necessarily cofibrantly generated. An adjunction $T \colon \M \rightleftarrows \K \colon S$ is an adjunction of awfs if and only if there is a lift \begin{equation}\label{Lcell}\xymatrix{ \J \ar[d] \ar@{-->}[r] & \Lalg \ar[d] \\ \M^\2 \ar[r]^{T^\2} & \K^\2}\end{equation} in which case the adjunction of awfs $(T,S,\l,\r): (\C,\F) \to (\L,\R)$ is canonically determined.
\end{thm}

 In other words, there is an adjunction of awfs $(\C,\F) \to (\L,\R)$ lifting $T\dashv S$ if an only if the image of the generators under $T^\2$ is \emph{cellular}. Write $T^\2\J$ for the category $\J$ over $\K^\2$. With this notation, the lifted functor of (\ref{Lcell}) is precisely a functor $T^\2\J \to \Lalg$ over $\K^\2$. 

\begin{proof}[Proof of Theorem \ref{cgadjthm}] In the presence of an adjunction of awfs $(T,S) \colon (\C,\F) \to (\L,\R)$, the functor \eqref{Lcell} is defined by composing the unit functor (I.2.26) with the lifted left adjoint. For the converse, a categorical expression of the familiar fact that adjunctions interact nicely with lifting problems is that \begin{equation}\label{adjpullback}\xymatrix{ (T^\2\J)^\boxslash \ar[d] \ar[r]^-{\mathrm{adj}} \ar@{}[dr]|(.2){\lrcorner} & \J^\boxslash \ar[d] \\ \K^\2 \ar[r]_{S^\2} & \M^\2}\end{equation} is a pullback in $\CAT$, or indeed in {\bf DblCAT}. The functor $\mathrm{adj}\colon (T^\2\J)^\boxslash \to \J^\boxslash$ sends an arrow $f$ with lifting function $\phi_f$ to the arrow $Sf$ with lifting function $\phi_f^\sharp$, whose chosen solutions are adjunct to the solutions chosen by $\phi_f$ to the transposed lifting problem. Define $\Ralg \to \Falg \cong \J^\boxslash$ to be the composite \[
\xymatrix{ \Ralg \ar[r]^-{\mathrm{lift}} & (\Lalg)^\boxslash \ar[r]^-{\mathrm{res}} & (T^\2\J)^\boxslash \ar[r]^-{\mathrm{adj}} & \J^\boxslash}\] where the restriction is along the functor $T^\2\J \to \Lalg$. Each functor preserves composition: the first by Lemma \ref{liftcomplem}, the second trivially, and the third by naturality of adjunctions---this last diagram chase is given in the proof of Theorem I.6.15. The conclusion follows from Lemma \ref{laxmorlem}. 
\end{proof}

The proof of Theorem \ref{cgadjthm} defines a canonical adjunction of awfs arising from a specified cellular structure for the generators $T^\2\J$. An immediate consequence of the following extension of the universal property of Theorem \ref{garnsthm} is that the cellular structure of $T^\2\J$ uniquely determines the adjunction of awfs.

\begin{thm}[I.6.22]\label{universalSOAthm} The unit functor \emph{(I.2.26)} constructed by Garner's small object argument is universal among adjunctions of awfs:  if $\J$ generates $(\C,\F)$ and $(\L,\R)$ is any awfs, a functor $\J \to \Lalg$ lifting a left adjoint factors through the unit along the lifted left adjoint of a unique adjunction of awfs $(\C,\F) \to (\L,\R)$.
\end{thm}

\subsection{The cellularity theorem}\label{twovarcellssec}

We can use the composition criterion of Theorems \ref{compcritthm} and \ref{twovarcompthm} to prove an analogous cellularity theorem for two-variable adjunctions of awfs. The full result, Theorem \ref{existthm} below, depends crucially on the single-variable case. 

\begin{thm}\label{cgadjawfsthm}  Suppose $\M$ has an awfs $(\C,\F)$ generated by $\J$ and $\K$ has an awfs $(\L,\R)$, not necessarily cofibrantly generated. Then $i \hast - \dashv \hhom(i,-)$ forms an adjunction of awfs $(\C,\F) \to (\L,\R)$  if and only if $i \hast \J$ is cellular, that is, if and only if there is a lift \[\xymatrix{ \J \ar@{-->}[r] \ar[d] & \Lalg \ar[d] \\ \M^\2 \ar[r]_{i\hast -} & \N^\2}\] 
\end{thm}

\begin{proof} As in the proof of Theorem \ref{cgadjthm}, we define the lift $\Ralg \to \Falg \cong \J^\boxslash$ of $\hhom(i,-)$ to be the composite  \[\xymatrix{ \Ralg \ar[r]^-{\mathrm{lift}} & \Lalg^\boxslash \ar[r]^-{\mathrm{res}} & (i \hast \J)^\boxslash \ar[r]^-{\mathrm{adj}} & \J^\boxslash}\] Explicitly, the image of an $\R$-algebra $f$ in $\J^\boxslash$ is the arrow $\hhom(i,f)$ equipped with a lifting function defined so that  the chosen solution $\phi_{\hhom(i,f)}(j, a, d \times c)$ to a lifting problem of the form displayed in the left-hand square of (\ref{crazycomp}) is adjunct to the solution constructed via the awfs $(\L,\R)$ and the functor $\J \to \Lalg$. 

By Lemma \ref{liftcomplem}, the functor $\Ralg \to (i \hast \J)^\boxslash$ preserves composition, so it suffices to show that adj$\colon (i \hast \J)^\boxslash \to \J^\boxslash$ satisfies the obvious analog of the criterion of Theorem \ref{compcritthm}. Given composable $(f, \phi_f^\sharp), (g,\phi_g^\sharp) \in (i \hast \J)^\boxslash$ their composite is $(gf, \phi_g^\sharp \bullet \phi_f^\sharp)$ where \[\phi_g^\sharp \bullet \phi_f^\sharp(i \hast j, c^\sharp \sqcup a^\sharp, b^\sharp) := \phi_f^\sharp(i \hast j, c^\sharp \sqcup a^\sharp, \phi_g^\sharp(i \hast j, fc^\sharp \sqcup fa^\sharp, b^\sharp))\] Transposing across the adjunction, we get the formula \[ \phi_{\hhom(i,gf)} (j, a, b \times c) := \phi_{\hhom(i,f)}(j, a, \phi_{\hhom(i,g)}(j, f^B a, b \times f^A  c) \times c)\] which says that the $\F$-algebra structure for $\hhom(i,gf)$ is obtained precisely as described in the statement of Theorem \ref{compcritthm}. Indeed, this is how that condition was discovered.
\end{proof}

We now extend this result to give a characterization of  two-variable adjunctions of awfs $(\C',\F') \times (\C,\F) \to (\L,\R)$ whose domain awfs are cofibrantly generated. The full classification is completed by the uniqueness theorem, proven below.

\begin{thm}[Cellularity Theorem]\label{existthm} Suppose $\I$ generates $(\C',\F')$ on $\K$ and $\J$ generates $(\C, \F)$ on $\M$ and $\N$ has an awfs $(\L,\R)$. Then $(\otimes,\homl,\homr)\colon \K \times \M \to \N$ forms a two-variable adjunction of awfs if and only if $\I \hast \J$ is cellular, that is, if and only if  there is a lift \[\xymatrix{ \I \times \J \ar[d] \ar@{-->}[r] & \Lalg \ar[d] \\ \K^\2 \times \M^\2 \ar[r]^-{-\hast-} & \N^\2}\]
\end{thm}

\begin{proof}
By Theorem \ref{cgadjawfsthm}, for each fixed $i \in \I$, the functor $i \hast - \colon \J \to \Lalg$ determines an adjunction of awfs $(i \hast -, \hhoml(i,-)) \colon (\C,\F) \to (\L,\R)$. A morphism $(a,b) \colon i' \To i$ in $\I$ induces a natural transformation $\hhoml(i,-) \To \hhoml(i',-)$ on the arrow categories. The lifts $\hhoml(i,-) \colon \Ralg \to \Falg$ assemble into a functor \[ \hhoml(-,-) \colon \I^\op \times \Ralg \to \Falg \] if and only if each component $\hhoml(i,f) \To \hhoml(i',f) \in \M^\2$ lifts to a morphism in $\Falg \cong \J^\boxslash$. If this is the case, it follows that the natural transformations $\r^\ell(i)$ characterizing each lifted functor $\hhoml(i,-)$ are natural in $\I$. By Lemma \ref{natmateslem}, their mates are then also natural in $\I$, and so the lifts of the left adjoints will assemble into a functor $\I \times \Calg \to \Lalg$, as in the proof of Theorem \ref{twovarcompthm}.

In other words, we must show that each lifted functor $\hhoml(i,-)$ assigns, to each $\R$-algebra $f$, solutions to all lifting problems between $j \in \J$ and $\hhoml(i,f)$ that are natural with respect to morphisms in $\J$ (so that this defines an object of $\J^\boxslash$), $\Ralg$ (so that this defines a functor), and $\I$ (so that the functors assemble into a bifunctor). The construction of Theorem \ref{cgadjawfsthm}, which solves the adjunct lifting problem using the functor $\I \times \J \to \Lalg$ and the awfs $(\L,\R)$, has all of these properties. 

To see this, note that the top composite below specifies the chosen solution to any lifting problem; in other words, this defines $\hhoml(i,f)$ as an element of $\J^\boxslash$. \begin{equation}\label{natsoln}\xymatrix@C=15pt{ \N^\2(j, \hhoml(i,f)) \ar[r]^-{\cong} \ar[d]|{\hhoml((a,b), f)_*} & \N^\2(i \hast j, f) \ar[d]^{ ((a,b) \hast j)^*} \ar[r]^-{\mathrm{solve}} & \N(B \otimes L, X) \ar[d]^{ (b \otimes L)^*} \ar[r]^-{\cong} & \M(L, \homl(B,X)) \ar[d]^{\homl(b,X)_*} \\ \N^\2(j, \hhoml(i', f)) \ar[r]_-{\cong} & \N^\2( i' \hast j, f) \ar[r]_-{\mathrm{solve}} & \N(B' \otimes L, X) \ar[r]_-{\cong} & \M(L, \homl(B',X))}\end{equation} Given $(a,b) \colon i' \To i$ in $\I$, the left square and right squares commute by naturality of the parameterized adjunctions in $\K^\2$ and $\K$. The essential point is that the middle square, whose horizontal arrows use the awfs $(\L,\R)$ to solve the lifting problem, also commutes, by functoriality of $\I \times \J \to \Lalg$ in the first variable and the fact that morphisms of $\L$-coalgebras preserve the chosen solutions to lifting problems against $\R$-algebras.

The left bottom composite of the rectangle chooses solutions to lifting problems against $\hhoml(i', f)$ that factor through lifting problems against $\hhoml(i,f)$. Commutativity of (\ref{natsoln}) asserts that these are the same lifts obtained by solving the lifting problem against $\hhoml(i,f)$ and then composing. This says exactly that the arrow in $\M^\2$ induced from $(a,b) \colon i' \To i$ lifts to $\J^\boxslash$, as desired.

We now use the lifted functor \begin{equation}\label{lastlift} \I \times \Calg \to \Lalg\end{equation} and repeat the argument just given. For each fixed $j \in \Calg$, the functor  $-\hast j \colon \I  \to \Lalg$ determines an adjunction of awfs \[(-\hast j, \hhomr(j,-)) \colon (\C',\F') \to (\L,\R)\] that depends also on the $\C$-coalgebra structure for $j$.  As above, the characterizing natural transformations are also natural in $\Calg$ and so the lifts $-\hotimes j \colon \Cpalg \to \Lalg$, $\hhomr(j,-)\colon \Ralg \to \Fpalg$ assemble into functors 
\begin{equation}\label{twothirds} -\hast - \colon \Cpalg \times \Calg \to \Lalg \end{equation}\[\hhomr(-,-) \colon \Calg^\op \times \Ralg \to \Fpalg.\]
Furthermore, their characterizing natural transformations are parameterized mates by the second half of the proof of Theorem \ref{twovarcompthm}.

The last step is subtle. We use the dual of the composition criterion of Theorem \ref{compcritthm} to show that for each $f \in \Ralg$, the lift \[\hhomr(-,f)\colon \Calg^\op \to \Fpalg\] obtained by restricting the second functor of (\ref{twothirds}) is a lax morphism of awfs. It follows from Theorem \ref{twovarcompthm} that the other parameterized mate of the natural transformations characterizing the functors (\ref{twothirds}) defines the final lifted functor \[\hhoml(-,-) \colon \Cpalg^\op \times \Ralg \to \Falg,\] completing the desired two-variable adjunction of awfs. 

In order to apply Theorem \ref{compcritthm}, we must show that given $j \colon I \to J$, $k \colon J \to K \in \Calg$, the unlabeled solutions that the functor $\hhomr(-,f)\colon \Cpalg^\op \to \Fpalg \cong \I^\boxslash$ assigns to the lifting problems below agree; for aesthetic reasons, we have abbreviated $\homr$ using exponential notation.
\begin{equation}\label{crazyopcomp} \xymatrix@R=40pt@C=40pt{ A \ar[r]^a \ar[d]_{\I \ni i} & X^K \ar[d]|(.3){\hhomr(k,f)} \ar@{=}[r] & X^K \ar[d]|(.6){\hhomr(kj,f)} \ar[r]^{X^k} & X^J \ar[d]|{\hhomr(j,f)} \\ B \ar@{-->}[ur] \ar@{-->}[urr] \ar@{-->}[urrr]_(.8)d \ar@{-->}[r]_(.4){b \times d} \ar@/_2pc/[rr]_{b \times c} &  Y^K \displaystyle\pb_{Y^J} X^J \ar[r]_{1 \times_{Y^j} X^j} & Y^K \displaystyle\pb_{Y^I} X^I \ar[r]_{Y^k \times 1} & Y^J \displaystyle\pb_{Y^I} X^I} \end{equation} 

The chosen lifts are defined by solving the adjunct lifting problems using the awfs $(\L,\R)$ and (\ref{lastlift}); transposing across the adjunction, it suffices to show that the unlabeled chosen solutions in the diagram 
\[\xymatrix@C=33pt@R=30pt{ A \otimes J \displaystyle\po_{A \otimes I} B \otimes I \ar[d]_{i \hast j} \ar[r]^-{A \otimes k \sqcup 1} & A \otimes K \displaystyle\po_{A \otimes I} B \otimes I \ar[d]_(.4){i \hast (kj)} \ar[r]^-{ 1 \sqcup_{A \otimes j} B \otimes j} \ar@/^2pc/[rr]^{a^\sharp \sqcup c^\sharp} & A \otimes K \displaystyle\po_{A \otimes J} B \otimes J \ar[d]_(.7){i \hast k} \ar[r]^(.65){a^\sharp \sqcup d^\sharp} & X \ar[d]^f \\ B \otimes J \ar[r]_{B \otimes k} \ar@{-->}[urrr]^(.2)d & B \otimes K \ar@{=}[r] \ar@{-->}[urr] & B \otimes K \ar[r]_{b^\sharp} \ar@{-->}[ur] & Y}\] agree.

If $j$ and $k$ have $\C$-coalgebra structures $s$ and $t$ and $f$ has $\R$-algebra structure $r$, the left-most unlabeled solution is defined to be \begin{equation}\label{leftthingone}\xymatrix@C=30pt{ B \otimes K \ar[r]^-{B \otimes (t \bullet s)} & B \otimes Q(kj) \ar[r]^-{\l(i)_{kj}} & E(i \hast (kj)) \ar[r]^-{E (a^\sharp \sqcup c^\sharp,b^\sharp)} & Ef \ar[r]^r & X}\end{equation} while the right-most is defined to be \begin{equation}\label{rightthingtwo}\xymatrix@C=30pt{ B \otimes K \ar[r]^-{B \otimes t} & B \otimes Qk \ar[r]^-{\l(i)_k} & E(i \hast k) \ar[r]^-{E(a^\sharp \sqcup d^\sharp, b^\sharp)} & Ef \ar[r]^r & X}\end{equation} where $d$, by naturality of $\l(i)$ with respect to the morphism $(1,k) \colon j \To kj$ of $\M^\2$, is
\[\xymatrix@C=25pt{ B \otimes J \ar[r]^-{B \otimes s} & B \otimes Qj \ar[r]^-{B \otimes Q(1,k)} & B \otimes Q(kj) \ar[r]^-{\l(i)_{kj}} & E (i \hast (kj)) \ar[r]^-{E(a^\sharp \sqcup c^\sharp, b^\sharp)} & Ef \ar[r]^r & X} \]  
We use this factorization of $d$ to factor the morphism $(a^\sharp \sqcup d^\sharp, b^\sharp) \colon i \hast k \To f$ of $\N^\2$ as \[\xymatrix@C=18pt{ A \otimes K \displaystyle\po_{A \otimes J} B \otimes J \ar[d]_{i \hast k} \ar[rr]^-*+{\labelstyle 1 \sqcup B \otimes [Q(1,k)s]} & & A \otimes K \displaystyle\po_{A \otimes Q(kj)} B \otimes Q(kj) \ar[d]_-{i \hast F(kj)} \ar[r]^-{\l'(i)_{kj}} & E(i \hast (kj)) \ar[d]_{R(i \hast (kj))} \ar[r]^-*+{\labelstyle E(a^\sharp \sqcup c^\sharp, b^\sharp)} & Ef \ar[d]_{Rf} \ar[r]^r & X \ar[d]_f \\ B \otimes K \ar@{=}[rr]& & B \otimes K \ar@{=}[r] & B \otimes K \ar[r]_{b^\sharp} & Y \ar@{=}[r] & Y}\] Applying the functor $E$ and substituting this factorization for $E(a^\sharp \sqcup d^\sharp, b^\sharp)$ in (\ref{rightthingtwo}),
\begin{align*} r\cdot &E(r,1) \cdot E(E(a^\sharp \sqcup c^\sharp, b^\sharp), b^\sharp) \cdot E(\l'(i)_{kj},1) \cdot E(1 \sqcup B \otimes Q(1,k)s,1) \cdot  \l(i)_k \cdot B \otimes t  \\ &= r \cdot \mu_f  \cdot  E(E(a^\sharp \sqcup c^\sharp, b^\sharp), b^\sharp) \cdot E(\l'(i)_{kj},1)  \cdot \l(i)_{F(kj)} \cdot B \otimes Q(Q(1,k)s,1) \cdot B \otimes t \\ &= r \cdot E(a^\sharp \sqcup c^\sharp, b^\sharp)  \cdot \mu_{i \hast (kj)} \cdot E(\l'(i)_{kj},1)  \cdot \l(i)_{F(kj)} \cdot B \otimes Q(Q(1,k)s,1)  \cdot B \otimes t \\ &= r \cdot E(a^\sharp \sqcup c^\sharp, b^\sharp) \cdot \l(i)_{kj} \cdot B \otimes \mu_{kj} \cdot B \otimes Q(Q(1,k)s,1)  \cdot B \otimes t \\ &= r \cdot E(a^\sharp \sqcup c^\sharp, b^\sharp) \cdot \l(i)_{kj} \cdot B \otimes (t\bullet s) \end{align*}
by naturality of $\l(i)$ and associativity of $r$, naturality of $\mu$, the monad pentagon for $\l(i)_{kj}$ which holds because $\l(i)$ defines a colax morphism of awfs, and the definition of $t \bullet s$. This last line is (\ref{leftthingone}), completing the proof.
\end{proof}

\subsection{Extending the universal property of the small object argument}\label{upssec}

In the remainder of this section, whenever we refer to an adjunction between arrow categories we always mean an adjunction of the form $T^\2 \dashv S^\2$ defined pointwise by an adjunction between the underlying categories, an adjunction of the form $i \hast - \dashv \hhom(i,-)$ defined by fixing one of the variables in a two-variable adjunction of the form \eqref{twovaradj}, or a composite of the two. We extend Theorem \ref{universalSOAthm} to the adjunctions of awfs of Definition \ref{newadjawfsdefn}. The uniqueness theorem, Theorem \ref{uniquethm}, is a corollary of this result. 

\begin{thm}\label{extendedUPthm} The unit functor \emph{(I.2.26)} constructed by Garner's small object argument is universal among adjunctions of awfs.
\end{thm}
\begin{proof}
Our argument extends the proof for Theorem \ref{universalSOAthm} given in I.6.22. We broaden our interpretation of the categories
\begin{equation}\label{semantics}{\small \xymatrix@C=17pt{\G^{\ladj}= \AWFS_{\ladj} \ar[r]^-{\G^{\ladj}_1} & \LAWFS_{\ladj} \ar[r]^-{\G^{\ladj}_2} & \Cmd(-)^{\2}_{\ladj} \ar[r]^-{\G^{\ladj}_3} & \CAT/(-)^{\2}_{\ladj}}}\end{equation}
of (I.6.21) and show that Garner's small object argument constructs a reflection along each forgetful functor. For each of these categories, the objects are the same as before, but we extend the class of morphisms to include those involving the sorts of adjunctions detailed above, always pointing in the direction of the left adjoint. A morphism in $\CAT/(-)^\2_\ladj$ is an adjunction between arrow categories together with a specified lift of the left adjoint to the fibers. A morphism in $\Cmd(-)^\2_\ladj$ is an adjunction between the arrow categories together with a specified colax comonad morphism over the left adjoint. $\LAWFS_\ladj$ is the full subcategory on comonads over the domain functor. $\AWFS_\ladj$ is the category of awfs and adjunctions of awfs.

Garner's small object argument constructs a reflection along $G_3^\ladj$ for the same reason as before:  left adjoints preserve left Kan extensions, regardless of how the adjunctions are defined. 

To apply the previous argument to demonstrate the reflection along $\G_2^\ladj$ in this setting, we must show that the functor $i \hast -$ preserves morphisms in the arrow category that are pushout squares in the underlying category. This follows from Lemma I.5.6 and the fact that the left adjoints $A \otimes -$ and $B\otimes-$ necessarily preserve pushouts. The rest of the argument is unchanged.

The final reflection along $\G_1^\ladj$ requires some work. The context for the argument of I.6.22 is the category $\FunF_\ladj$ whose objects are functorial factorizations and whose morphisms are colax morphisms of functorial factorizations lifting left adjoints. Because functors of the form $i \hast -$ preserve neither domains nor composability, a colax morphism of functorial factorizations $(i \hast -, \l(i)) \colon \vec{Q} \to \vec{E}$  now has the form displayed in the left-hand diagram (\ref{gammarhodefn}). These colax morphisms of functorial factorizations compose with those of Definition \ref{defn:colaxfunfact}, so it suffices to consider only those colax morphisms lifting functors the form $i \hast -$, the other case completed in the original proof. 

To apply the argument of I.6.22, we must show that the category $\FunF_\ladj$ has the following two properties. Each fiber, that is, each category of functorial factorizations on a fixed category, has two monoidal structures $\oast$ and $\odot$, given by re-factoring the right or the left factor, respectively. 
We must show \begin{itemize} \item a pair of morphisms $\phi, \psi$ lifting the same left adjoint $i \hast -$ can be combined to give $\phi \oast \psi$ and $\phi \odot \psi$ \item the distributive law $\a$ of \cite[\S 3.2]{garnercofibrantly} is natural with respect to colax morphisms lifting $i \hast -$, i.e., the following diagram commutes:
\[\xymatrix{ (\vec{X} \odot \vec{X'}) \oast (\vec{Z} \odot \vec{Z'}) \ar[r]^{\a} \ar[d]|{(\phi \odot \phi') \oast (\psi \odot \psi')} & (\vec{X} \oast \vec{Z}) \odot (\vec{X'} \oast \vec{Z'}) \ar[d]|{(\phi \oast \psi) \odot (\phi' \oast \psi')} \\ (\vec{Y} \odot \vec{Y'}) \oast (\vec{W} \odot \vec{W'}) \ar[r]_{\a} & (\vec{Y} \oast \vec{W} ) \odot (\vec{Y'} \oast \vec{W'})}\] \end{itemize} It follows that if $\phi$ and $\psi$ are in $\LAWFS_\ladj$, that is if $\phi$ and $\psi$ are $\odot$-comonoid morphisms, then so is $\phi \oast \psi$. 

We define the products $\phi \oast \psi$ and $\phi \odot \psi$ of colax morphisms lifting $i \hast -$ and leave the tedious but straightforward diagram chase exhibiting the distributive law to the reader. Given functorial factorizations $\vec{Q}=(C,F)$, $\vec{Q^*}\!=(C^*,F^*)$ on $\M$ and $\vec{E}=(L,R)$, $\vec{E^*}\!=(L^*,R^*)$ on $\N$ together with morphisms $\phi \colon \vec{Q} \to \vec{E}$ and $\psi \colon \vec{Q^*} \to \vec{E^*}$ lifting $i\hast-$, $\phi \oast \psi$ is the composite $E(\psi',1)\cdot \phi_{F^*}$ displayed below 
\[\xymatrix@C=95pt@R=40pt@!0{ **[r]{B \otimes  J\! \displaystyle\po_{A \otimes  J}\! A \otimes QF^*\!j} \ar[dr]^{B \otimes C^*\!j \sqcup 1} \ar[dd]^(.65){i \hast (CF^*\! \cdot C^*\!)j}  \ar[rrr]^{1 \sqcup A \otimes FF^*\!j} & & & **[l]{B \otimes  J\! \displaystyle\po_{A \otimes  J}\! A \otimes  K \ar[d]_{L^*\!(i \hast j)} \ar[dl]_{B \otimes C^*\!j \sqcup 1}} \\ & **[l]{B \otimes Q^*\!j\! \displaystyle\po_{A \otimes Q^*\!j}\! A \otimes QF^*\!j{\quad} \ar[r] } \ar[dl]^{i \hast CF^*\!j} & **[r]{B \otimes Q^*\!j\! \displaystyle\po_{A \otimes Q^*\!j}\! A \otimes  K} \ar[r]_-{\psi'_j}  \ar[d]_{L(i \hast F^*\!j)} & **[l]{E^*\!( i \hast j)} \ar[d]_{LR^*\!(i \hast j)} \\ **[r]{B \otimes QF^*\!j} \ar[rr]^{\phi_{F^*\!j}} \ar[d]^{\iota} & & E (i \hast F^*\!j) \ar[dd]_{R(i \hast F^*\!j)} \ar[r]^-{E(\psi'_j,1)} & **[l]{ER^*\!(i \hast j)} \ar[dd]_{RR^*\!(i \hast j)} \\ **[r]{B \otimes QF^*\!j\! \displaystyle\po_{A \otimes QF^*\!j}\! A \otimes  K} \ar[d]^{i \hast FF^*\!j} \ar[urr]_-{\phi'_{F^*\!j}} \\ **[r]{B \otimes  K} \ar@{=}[rr] & & B \otimes  K \ar@{=}[r] & **[l]{B \otimes  K}}\]
Similarly, $\phi \odot \psi$ is the composite $E(1 \sqcup A \otimes F^*\!,\psi) \cdot \phi_{C^*\!}$ displayed below
\[\xymatrix@C=90pt@R=40pt@!0{ **[r]{B \otimes  J \displaystyle\po_{A \otimes  J} A \otimes QC^*\!j} \ar[dd]^{i \hast CC^*\!j} \ar[drr]_{1 \sqcup A \otimes FC^*\!j} \ar[rrr]^{1 \sqcup A \otimes (F^*\! \cdot FC^*\!)j} & &  & B \otimes  J \displaystyle\po_{A \otimes  J} A \otimes  K \ar[dd]^{LL^*\!(i \hast j)} \\ &&  B \otimes  J \displaystyle\po_{A \otimes  J} A \otimes Q^*\!j \ar[ur]_{1 \sqcup A\otimes F^*\!j} \ar[d]^{L(i \hast C^*\!j)} \\ **[r]{B \otimes QC^*\!j} \ar[dd]^{\iota} \ar[dr]^{\iota} \ar[rr]^{\phi_{C^*\!j}} & & E(i \hast C^*\!j)  \ar[dd]^{R(i \hast C^*\!j)} \ar[r]_{E (1 \sqcup A \otimes F^*\!j, \psi)} & EL^*\! (i \hast j) \ar[dd]^{RL^*\!(i \hast j)} \\ & B \otimes QC^*\!j \displaystyle\po_{A \otimes QC^*\!j} A \otimes Q^*\!j \ar[dl]^(.4){1 \sqcup A \otimes F^*\!j} \ar[ur]^{\phi'_{C^*\!j}} \ar[dr]^-{i \hast FC^*\!j}&  \\ **[r]{B \otimes QC^*\!j \displaystyle\po_{A \otimes QC^*\!j} A \otimes  K}   \ar[d]^{i \hast (F^*\! \cdot FC^*\!)j} & & B \otimes Q^*\!j \ar[r]_{\psi_{j}} \ar[dll]^{B \otimes F^*\!j} & E^*\!(i \hast j) \ar[d]^{R^*\!(i \hast j)} \\ **[r]{B \otimes  K} \ar@{=}[rrr] & && B \otimes K  }\]
\end{proof}

The converse to Theorem \ref{existthm} follows as a corollary.

\begin{thm}[Uniqueness Theorem]\label{uniquethm} Fix a two-variable adjunction and awfs as in \ref{terriblentn}. If $\I$ and $\J$ generate $(\C',\F')$ and $(\C,\F)$, there is at most one two-variable adjunction of awfs $(\C',\F') \times (\C,\F) \to (\L,\R)$ whose lifted left adjoint restricts along the unit functors to a given $\I \times \J \to \Lalg$.
\end{thm}
\begin{proof}
Suppose given a pair of two-variable adjunctions of awfs \[\Cpalg \times \Calg \rightrightarrows \Lalg\] extending $\I \times \J \to \Lalg$. On morphisms their behavior is completely specified by the condition that they lift $-\hast-$, so it suffices to consider whether these functors agree at each pair of objects. Restricting along the unit $\I \to \Cpalg$, we obtain a pair of functors $\I \times \Calg \rightrightarrows \Lalg$ necessarily distinct: if they agreed for each $j \in \Calg$, their extensions at each $j$, the adjunctions of awfs $\Cpalg \rightrightarrows \Lalg$, would also agree by Theorem \ref{extendedUPthm}. Now restricting these functors along the unit $\J \to \Calg$ we obtain, in both cases, the original $\I \times \J \to \Lalg$, by hypothesis. But this contradicts the argument just given: at each $i \in \I$, the extension to an adjunction of awfs $\Calg \to \Lalg$ is unique by Theorem \ref{extendedUPthm}. Thus, there can be at most one functor $\I \times \Calg \to \Lalg$, and hence at most one $\Cpalg \times \Calg \to \Lalg$ extending $\I \times \J \to \Lalg$.
\end{proof}

\section{Algebraic Quillen two-variable adjunctions}\label{algQuillsec}

An algebraic model category is a homotopical category equipped with a pair of interacting algebraic weak factorization systems and a comparison morphism. Hence, algebraic left and right Quillen functors must take into account these interactions. In this section, we extend our notions of morphisms of awfs to the model category setting, paving the way for the introduction of monoidal algebraic model structures in the next section.

\subsection{Algebraic model structures}

A \emph{homotopical category} $(\M,\W)$ is a complete and cocomplete category $\M$ together with a class of morphisms $\W$ called \emph{weak equivalences} that satisfy the 2-of-3 property. 

\begin{defn} A \emph{model structure} on a homotopical category $(\M,\W)$ consists of two classes of morphisms $\cC, \cF$ such that $(\cC \cap \W, \cF)$ and $(\cC, \cF \cap \W)$ are weak factorization systems.
\end{defn}

See \cite[\S 7]{joyaltierneyquasi} or \cite[\S 14.2]{maypontoconcise2} for proof that this definition agrees with the usual one.

\begin{defn} An \emph{algebraic model structure} on a homotopical category $(\M,\W)$ consists of a morphism of algebraic weak factorization systems $(\C_t,\F) \to (\C,\F_t)$ so that the underlying weak factorization systems $(\cC_t,\cF)$ and $(\cC,\cF_t)$ form a model structure on $\M$ with weak equivalences $\W$.
\end{defn}

An essential application of the universal properties of Theorem \ref{garnsthm} is:

\begin{thm}[I.3.6]\label{amsexistcor} 
An ordinary cofibrantly generated model structure with generating trivial cofibrations $\J$ and generating cofibrations $\I$ on a category permitting the small object argument has an algebraic model structure with the same generators if and only if the elements of $\J$ are $\I$-cellular, i.e., if and only if there is a functor $\J \to \Calg$ over $\M^\2$.
\end{thm}
\begin{proof}
Given such an algebraic model structure $(\C_t,\F) \to (\C,\F_t)$, the functor $\Ctalg \to \Calg$ determined by the comparison map defines $\C$-coalgebra structures for the generating trivial cofibrations. Conversely, let $(\C_t,\F)$ and $(\C,\F_t)$ denote the awfs generated by $\J$ and $\I$. Given $\J \to \Calg$, by (I.2.26) this functor factors through the unit $\J \to \Ctalg$ along a functor induced by a morphism of awfs $(\C_t,\F) \to (\C,\F_t)$. On account of the isomorphisms $\Falg \cong \J^\boxslash$, $\Ftalg \cong \I^\boxslash$ of (I.2.27) the underlying wfs of the awfs $(\C_t,\F)$ and $(\C,\F_t)$ coincide with the wfs in the ordinary model structure generated by $\J$ and $\I$. Hence, this algebraic model structure is compatible with the original model structure.
\end{proof}

\begin{ex}\label{ex:ssets} Quillen's model structure on simplicial sets is an algebraic model structure generated by sets $\I$ and $\J$ of sphere and horn inclusions. A horn inclusion $\Lambda^n_k \to \Delta^n$ factors through $\partial\Delta^n$; the first factor is a pushout of a map in $\I$ and the second factor is an element of $\I$. Both factors and hence their composite are canonically $\C$-coalgebras. In this way, we see that elements of $\J$ are cellular; the conclusion follows from Theorem \ref{amsexistcor}.
\end{ex}

\begin{rmk}
In fact, any cofibrantly generated model structure on a category permitting the small object argument gives rise to an algebraic model structure even if the elements of $\J$ aren't $\I$-cellular, though at the cost of changing one of the generating sets. See I.3.7 and I.3.8.
\end{rmk}

\subsection{Algebraic Quillen adjunctions}\label{algquillssec}

An \emph{algebraic Quillen adjunction} $T\colon \M \rightleftarrows \K \colon S$ between categories equipped with algebraic model structures consists of adjunctions of awfs with respect to the (trivial cofibration, fibration) and (cofibration, trivial fibration) awfs satisfying an additional compatibility condition.

\begin{defn}[I.3.11] Suppose $\M$ and $\K$ are categories with algebraic model structures $\xi^{\M} \colon (\C_t,\F) \to (\C,\F_t)$ and  $\xi^\K \colon (\L_t,\R) \to (\L, \R_t)$. An \emph{algebraic Quillen adjunction} is an adjunction $T  \colon \M \rightleftarrows \K\colon S$ together with adjunctions of awfs \[\xymatrix@C=50pt{  (\C_t, \F) \ar[dr]|-{(T,S)} \ar[r]^{(T,S)} \ar[d]_{\xi^{\M}} & (\L_t,\R) \ar[d]^{\xi^{\K}} \\ (\C,\F_t) \ar[r]_{(T,S)}   & (\L,\R_t) }\] such that both triangles commute.
\end{defn}
In particular, an algebraic Quillen adjunction consists of commuting lifted double functors
\begin{equation}\label{quilldbl}\raisebox{.25in}{\xymatrix{\Alg(\R_t) \ar[d]_{\xi^\K} \ar[r]^{S^\2} & \Alg(\F_t) \ar[d]^{\xi^\M} \\ \Alg(\R) \ar[r]_{S^\2} & \Alg(\F)}} \quad \mathrm{and} \quad \raisebox{.25in}{\xymatrix{ \Coalg(\C_t) \ar[d]_{\xi^\M} \ar[r]^{T^\2} & \Coalg(\L_t) \ar[d]^{\xi^\K} \\ \Coalg(\C) \ar[r]_{T^\2} & \Coalg(\L)}}\end{equation}
This compatibility condition is equivalent to (I.3.12), which asks that the ordinary lifted functors on algebraic (trivial) cofibrations and fibrations commute. Taking either perspective, functors on the left-hand or right-hand sides determine those on the other. In particular, it suffices to check commutativity of one of these two diagrams. For example:

\begin{thm} Suppose $\M$ has an algebraic model structure $\xi \colon (\C_t,\F) \to (\C,\F_t)$. Then the category $\M_*$ of pointed objects in $\M$ has an algebraic model structure such that the disjoint basepoint--forgetful adjunction $(-)_+ \dashv U \colon \M \rightleftarrows \M_*$ is an algebraic Quillen adjunction.
\end{thm}
\begin{proof} The category $\M_*$ is isomorphic to the slice category $*/\M$, where $*$ denotes the terminal object. An arrow or a commutative square in $\M_*$ is determined by the arrow or square in the image of the forgetful functor together with the basepoint of its initial object; the other basepoints are defined by composition. This says that \[\xymatrix{ (\M_*)^\2 \ar[r]^{U^\2} \ar[d]_{\dom} \ar@{}[dr]|(.2){\lrcorner} & \M^\2 \ar[d]^\dom \\ \M_* \ar[r]_U & \M}\] is a pullback. We will see that this implies that the algebraic model structure on $\M$ can be lifted along $U$ to define an algebraic model structure on $\M_*$.

The comonad $\C$ is domain-preserving, so its constituent functor and natural transformations can be pulled back to $(\M_*)^\2$; this works for the 2-cells because limits in $\CAT$ are also 2-limits \cite{kellyelementary}. \[\xymatrix{ (\M_*)^\2 \ar[rr]^{U^\2} \ar@{-->}[dr]^{C_*} \ar@/_/[ddr]_{\dom} & & \M^\2 \ar[dr]^C \ar@/_/[ddr]|(.47){\hole}_(.67){\dom} \\ & (\M_*)^\2 \ar[d]^{\dom} \ar@{}[drr]|(.2){\lrcorner} \ar[rr]^{U^\2} && \M^\2 \ar[d]^\dom \\ & \M_* \ar[rr]_U & & \M}\] 
The multiplication for the monads also lifts to $\M_*$: e.g., the basepoint of $FRf$ is the image of the basepoint of $\dom f$, which maps to the basepoint of $Rf$, which proves that $\mu_f$ preserves basepoints. For similar reasons, the comparison map lifts along $U$. This defines an algebraic model structure we denote $\xi_* \colon ((\C_t)_*,\F_*) \to (\C_*, (\F_t)_*)$ on $\M_*$.

Algebra structures for fibrations in $\M_*$ are precisely algebra structures for the underlying fibrations in $\M$: the basepoint of $Rf$ is in the image of the basepoint of $\dom f$ and hence maps via the algebra structure map back to the basepoint of $\dom f$. It follows that the left-hand diagram \[\xymatrix{ \Alg(\F_*) \ar[r] \ar[d] \ar@{}[dr]|(.2){\lrcorner} & \Alg(\F) \ar[d] \\ \SSq(\M_*) \ar[r]_U & \SSq(\M)} \quad \quad \xymatrix{ \Alg((\F_t)_*) \ar[d]_{\xi_*} \ar[r] & \Alg(\F_t) \ar[d]^{\xi} \\ \Alg(\F_*) \ar[r] & \Alg(\F)} \] is a pullback in {\bf DblCAT}. By this fact and the definition of $\xi_*$, the right-hand square commutes, establishing the algebraic Quillen adjunction.
\end{proof}

By Theorems \ref{cgadjthm} and \ref{universalSOAthm}, the compatibility conditions \eqref{quilldbl} can be tested at the level of generating (trivial) cofibrations.

\begin{thm}\label{badthm}
Suppose that $\M$ and $\K$ have algebraic model structures, as above, such that the algebraic model structure on $\M$ is generated by categories $\J$ and $\I$. Then $T  \colon \M \rightleftarrows \K \colon S$ is an algebraic Quillen adjunction if and only if there exist commuting lifts \[\xymatrix@=10pt{\J \ar[dd] \ar@{-->}[rr] & & \Ltalg \ar'[d][dd] & & \J \ar[dr] \ar[dd] \ar[rr] & & \Ltalg \ar'[d][dd] \ar[dr]^{\xi^{\K}}  \\ & \I \ar[dl] \ar@{-->}[rr] & & {\Lalg} \ar[dl] &  & \Calg \ar[dl] \ar[rr] & & **[l]{\Lalg} \ar[dl] \\ \M^\2 \ar[rr]_{T^\2} & & \K^\2 & & \M^\2 \ar[rr]_{T^\2} & & \K^\2}\] 
in which case the algebraic Quillen adjunction is canonically determined.
\end{thm}

The first condition says that the images of $\J$ and $\I$ must be cellular for $\L_t$ and $\L$ respectively. The second condition says that the two canonical ways of assigning $\L$-coalgebra structures to $\J$---one using $\xi^\M$ and one lifted functor and the other using $\xi^\K$ and the other lifted functor---must agree.

\begin{proof}[Proof of Theorem \ref{badthm}] By Theorem \ref{cgadjthm}, the lifts of $T^\2$ give rise to adjunctions of awfs \[(T,S) \colon (\C_t,\F) \to (\L_t,\R) \quad\quad (T,S) \colon (\C,\F_t) \to (\L,\R_t).\] These combine to specify an algebraic Quillen adjunction if and only if the lifted functors  \begin{equation}\label{coalglifts}\xymatrix@=10pt{\Ctalg \ar[dr]^(.6){\xi^\M} \ar[dd] \ar[rr] & & \Ltalg \ar[dr]^{\xi^\K}\ar'[d][dd] \\ & \Calg \ar[dl] \ar[rr] & & \Lalg \ar[dl] \\ \M^\2 \ar[rr]_{T^\2} & & \K^\2}\end{equation} commute. The functor $\Ctalg \to \Ltalg$ is defined by factoring $\J \to \Ltalg$ through $\Ctalg$ using the universal property of Theorem \ref{universalSOAthm}. Again by the universal property of $\J \to \Ctalg$, (\ref{coalglifts}) commutes if and only if the restriction to $\J$ does, which was a hypothesis.
\end{proof}

In particular, the conditions of Theorem \ref{badthm} are satisfied if the algebraic model structure on $\K$ is constructed by lifting the algebraic model structure on $\M$ along an adjunction 

\begin{thm}[I.3.10, I.3.13] Suppose $\M$ has an algebraic model structure generated by $\J$ and $\I$, $T  \colon \M \rightleftarrows \K \colon S$ is an adjunction, and $\K$ permits the small object argument. If \begin{itemize} \item[($\dagger\dagger$)] $S$ maps the $T^\2\J$-cellular arrows into weak equivalences \end{itemize} then $T^\2\J$ and $T^\2\I$ generate an algebraic model structure on $\K$ such that $T \dashv S$ is canonically an algebraic Quillen adjunction.
\end{thm}

This gives an important class of algebraic Quillen adjunctions, including the geometric realization--total singular complex adjunction between simplicial sets and spaces, the adjunction between $G$-spaces and space-valued presheaves on the orbit category for a group $G$, the adjunctions establishing a projective model structure, as well as many other classical examples.

\subsection{Algebraic Quillen two-variable adjunctions}
 
If $\K$, $\M$, and $\N$ are model categories, the two-variable adjunction (\ref{twovaradj1}) is \emph{Quillen} if the following equivalent conditions are satisfied \cite{hoveymodel}: \begin{itemize} \item[(a)] $\otimes$ is a \emph{left Quillen bifunctor}: if $i \in \K^\2$ and $j \in \M^\2$ are cofibrations then $i\hast j \in \N^\2$ is a cofibration that is trivial if either $i$ or $j$ is \item[(b)]  $\homl$ is a \emph{right Quillen bifunctor}: if $i \in \K^\2$ is a cofibration and $f \in \N^\2$ is a fibration then $\hhoml(i,f) \in \M^\2$ is a fibration that is trivial if either $i$ or $f$ is \item[(c)] $\homr$ is a \emph{right Quillen bifunctor}: if $j \in \M^\2$ is a cofibration and $f \in \N^\2$ is a fibration then $\hhomr(j,f) \in \K^\2$ is a fibration that is trivial if either $j$ or $f$ is \end{itemize}
The equivalence of the three conditions rests on the interplay between adjunctions and lifting problems.
This should be thought of as a strengthening of the usual lifting axiom. For instance, the corresponding axiom (c) for simplicial model categories implies that any two solutions to a lifting problem under a cofibrant object are homotopic relative to that object \cite[\S II.3]{goerssjardinesimplicial}. 

A two-variable adjunction $(\otimes,\homl,\homr)$ is \emph{algebraic Quillen} if the two-variable adjunction $(\hast, \hhoml, \hhomr)$ lifts to functors of algebraic (trivial) cofibrations and fibrations as appropriate. The symmetry of the classical setting---the equivalence of conditions (a), (b), and (c)---is captured by the requirement that the parameterized mates of the natural transformation characterizing the lift of one of the functors $(\hast, \hhoml, \hhomr)$ characterize the others. 

\begin{defn}\label{algquilltwovardefn} 
Suppose $\K$, $\M$, and $\N$ have algebraic model structures \[ \xi^\K \colon (\C_t',\F') \to (\C',\F_t'),\hspace{.2cm}  \xi^\M \colon (\C_t,\F) \to (\C,\F_t),\hspace{.2cm}  \mathrm{and}\hspace{.2cm} \xi^\N \colon (\L_t,\R) \to (\L,\R_t).\] 
An \emph{algebraic Quillen two-variable adjunction} \[(\otimes, \homl, \homr) \colon \K \times \M \to \N\] consists of specified two-variable adjunctions of awfs \begin{align*}  & \otimes \colon (\C',\F_t') \times (\C,\F_t) \to (\L,\R_t)\\ &\otimes \colon (\C_t',\F') \times (\C,\F_t) \to (\L_t,\R) \\ &\otimes \colon (\C',\F_t') \times (\C_t,\F) \to (\L_t,\R)\end{align*} The algebraic Quillen two-variable adjunction is \emph{maximally coherent} if the lifted functors \begin{equation}\label{algbiquill}\xymatrix@!0@C=50pt@R=15pt{ & \Cptalg  \times \Ctalg \ar[dddr]^{\xi^\K \times 1} \ar@{-->}[ddrrr]  \ar[ddl]_(.7){1 \times \xi^\M}\\ \\ \Cptalg \times \Calg  \ar[dddr]_(.4){\xi^\K \times 1} \ar[rrrr]|(.41){\hole} & & & &  \Ltalg \ar[dddr]^{\xi^\N} \\  & & \Cpalg \times \Ctalg \ar[urr] \ar[ddl]^(.3){1 \times \xi^\M} \\  \\& \Cpalg \times \Calg \ar[rrrr] & & && \Lalg}\end{equation}
commute.
\end{defn} 

The condition (\ref{algbiquill}) asks that three squares relating each pair of two-variable adjunctions of awfs commute. By the calculus of parameterized mates, the coherence conditions (\ref{algbiquill}) are equivalent to coherence conditions for the lifts of $\hhoml$ or $\hhomr$. The proof requires the following lemma.

\begin{lem}\label{complem} Two-variable adjunctions of awfs can be composed with adjunctions of awfs (pointing in the correct direction) in any of the variables to obtain another two-variable adjunction of awfs.
\end{lem}
\begin{proof} The functors lifting the left adjoints can clearly be composed; unpacking Lemma \ref{appellem}, the natural transformation characterizing the composite is a pasted composite of the natural transformations characterizing each piece. By Lemma \ref{paracomplem} and the calculus of parameterized mates, the parameterized mates of this composite natural transformation are obtained by pasting the mates of characterizing natural transformations, and hence characterize the functors obtained by composing the appropriate right adjoints. So we see that the composite is again a two-variable adjunction of awfs. 
\end{proof}

Note that a maximally coherent algebraic Quillen two-variable adjunction also specifies  a fourth two-variable adjunction of awfs $\otimes \colon (\C_t',\F') \times (\C_t,\F) \to (\L_t,\R)$ whose lifted functor is the dotted arrow of \eqref{algbiquill}. 

\begin{cor}\label{algbiquillcor} The lifted functors (\ref{algbiquill}) commute if and only if the lifts 
\begin{equation}\label{algbiquillhom}\xymatrix@!0@C=50pt@R=15pt{ & \Cptalg  \times \Rtalg \ar[dddr]^{\xi^\K \times 1} \ar@{-->}[ddrrr]  \ar[ddl]_(.7){1 \times \xi^\N} \\ \\ \Cptalg \times \Ralg  \ar[dddr]_(.4){\xi^\K \times 1} \ar[rrrr]|(.41){\hole} & & & &  \Ftalg \ar[dddr]^{\xi^\M} \\  & & \Cpalg \times \Rtalg \ar[urr] \ar[ddl]^(.3){1 \times \xi^\N} \\  \\& \Cpalg \times \Ralg \ar[rrrr] & & && \Falg}\end{equation} of $\hhoml$ commute, and similarly for $\hhomr$. 
\end{cor}
\begin{proof} A parameterized mate of the composite two-variable adjunction of awfs defined by each commuting square of (\ref{algbiquill}) characterizes the corresponding commuting square of (\ref{algbiquillhom}). 
\end{proof}

Evaluating a maximally coherent algebraic Quillen two-variable adjunction at an algebraic cofibrant object or an algebraic fibrant object gives rise to an ordinary algebraic Quillen adjunction.

\begin{lem} If $(\otimes,\homl,\homr) \colon \K \times \M \to \N$ is a maximally coherent algebraic Quillen two-variable adjunction and $A$ is an algebraic cofibrant object of $\K$, then $A \otimes -\colon \M \rightleftarrows \N \colon  \homl(A,-) $ is canonically an algebraic Quillen adjunction. Dually, if $X$ is an algebraic fibrant object of $\N$, then $\homl(-,X) \colon \K \rightleftarrows \M^\op \colon  \homr(-,X)$ is canonically an algebraic Quillen adjunction.
\end{lem}
\begin{proof}
Using the notation of Definition \ref{algquilltwovardefn}, an algebraic cofibrant object $A$ is a $\C'$-coalgebra $i \colon \emptyset \to A$. The adjunction $i \hast - \dashv \hhoml(i,-)$ coincides with the pointwise-defined adjunction \[A \otimes - \colon \M^\2 \rightleftarrows \N^\2 \colon \homl(A,-).\] Hence, upon evaluating at $i \in \Cpalg$, the front rectangle of (\ref{algbiquill}) exhibits the desired algebraic Quillen adjunction.
\end{proof}

In analogy with Theorem \ref{badthm}, when the domain algebraic model structures are cofibrantly generated, the cellularity and uniqueness theorems give a characterization of algebraic Quillen two-variable adjunctions.

\begin{cor} Suppose the algebraic model structures on $\K$ and $\M$ are cofibrantly generated, with generating categories $\J'$, $\I'$, $\J$, and $\I$. Then $(\otimes,\homl, \homr)$ is an algebraic Quillen two-variable adjunction if and only if the category $\I' \times \I$ is $\L$-cellular and the categories $\J' \times \I$ and $\I' \times \J$ are $\L_t$-cellular, and is maximally coherent if and only if the following diagrams commute.
\[\xymatrix@C=35pt@R=20pt@!0{ & \I' \times \J \ar[drr]  \ar[dddl]|(.35){\hole} \\\J' \times \I \ar[dd] \ar[rrr] & & & \Ltalg \ar[dd] \\ \\ \save[]{\Cpalg \times \Calg}\restore \ar[rrr] && & \Lalg} \xymatrix@C=30pt@R=20pt@!0{ \J' \times \J \ar[dd] \ar[rrr] && & \save[]{\Cptalg \times \Calg}\restore \ar[dd] \\  \\ \save[]{\Cpalg \times \Ctalg}\restore \ar[rrr]&& & **[l]{\Ltalg}  }\]
\end{cor}

\section{Monoidal algebraic model structures}\label{monsec}

We are finally in a position to introduce the main definition.

\begin{defn}\label{monamsdefn} A \emph{monoidal algebraic model structure} on a closed monoidal category $(\otimes,\homl,\homr) \colon \M \times \M \to \M$ with monoidal unit $\1$ is an algebraic model structure $\xi \colon (\C_t,\F) \to (\C,\F_t)$
such that \begin{enumroman}\item $(\otimes, \homl, \homr)$ is an algebraic Quillen two-variable adjunction \item tensoring on either side with $\e_\1 \colon Q\1 \to \1$, the cofibrant replacement comonad counit, sends cofibrant objects to weak equivalences. \end{enumroman}
\end{defn}

Monoidal algebraic model categories are in particular monoidal model categories in the sense of \cite{hoveymodel}. It makes no difference whether condition (ii) is stated for algebraic cofibrant objects or ordinary cofibrant objects. If the unit $\1$ is cofibrant, (ii) is automatic from (i) and Ken Brown's lemma.

In the case where the monoidal structure is symmetric, a two-variable adjunction of awfs $(\C_t,\F) \times (\C,\F_t) \to (\C_t,\F)$ gives rise to a two-variable adjunction of awfs $(\C,\F_t) \times (\C_t,\F) \to (\C_t,\F)$ by composing with the symmetry isomorphism. When $(\C,\F_t)$ is generated by $\I$, Theorem \ref{uniquethm} implies that the two-variable adjunction of awfs $(\C,\F_t) \times (\C,\F_t) \to (\C,\F_t)$ commutes with the symmetry isomorphism if and only if the functor $\I \hast \I \to \Calg$ is defined symmetrically. Thus:

\begin{defn} A \emph{symmetric monoidal algebraic model structure} on a closed symmetric monoidal category $(\otimes,\hom,\hom) \colon \M \times \M \to \M$ with monoidal unit $\1$ is an algebraic model structure such that \begin{enumroman}\item $(\otimes, \hom, \hom)$ is an algebraic Quillen two-variable adjunction such that the lifted functors of algebraic (trivial) cofibrations commute up to isomorphism with the symmetry isomorphism
 \item tensoring with $\e_\1 : Q\1 \to \1$ sends cofibrant objects to weak equivalences \end{enumroman}
\end{defn}

We now use Theorems \ref{twovarcompthm}, \ref{existthm}, and \ref{uniquethm} to find examples.

\begin{thm}\label{folkthm} The folk model structure on $\Cat$ is a maximally coherent symmetric monoidal algebraic model structure.
\end{thm}
\begin{proof}
The folk model structure on $\Cat$ is generated by the following sets of functors \[ \I = \left\{ \raisebox{20pt}{\xymatrix{ \emptyset \ar@{|->}[d]_c \\ \bullet }},  \raisebox{20pt}{\xymatrix@C=5pt{ \bullet &  \ar@{ |->}[d]_d & \bullet  \\ \bullet \ar[rr] & {~} &  \bullet }} , \raisebox{20pt}{\xymatrix@C=5pt{ \bullet \ar@<.5ex>[rr] \ar@<-.5ex>[rr] & \ar@{ |->}[d]_e & \bullet \\ \bullet \ar[rr] & {~} & \bullet }}  \right\} \quad\quad \J = \left\{ \raisebox{20pt}{\xymatrix@C=15pt{ & \bullet \ar@{|->}[d]_{j} \\ \bullet \ar@<.5ex>[r] & \bullet \ar@<.5ex>[l] }}\right\} \] Write $\iso$ for the codomain of $j$, that is, the free-standing isomorphism

By Theorem \ref{amsexistcor}, $\Cat$ has an algebraic model structure generated by $\I$ and $\J$ if and only if $j$ is $\I$-cellular, i.e., if and only if there is a functor $\J \to \Calg$, where $\C$ is the comonad of the awfs generated by $\I$. The comonad $\C$ is particularly simple to describe. By a dimension argument, it can be constructed by running Garner's small object argument first using the generator $c$, then using $d$, and then using $e$. Each process converges after a single step, which means that the comonad $\C$ is constructed in three steps: each of which forms a single pushout of the coproduct over lifting problems against the generator in question.  See \cite[\S 4]{garnerunderstanding} or \S I.2.5 for more details about the small object argument.

The resulting functorial factorization is equivalent to the usual mapping cylinder construction: $$\xymatrix{ & A \ar[r]^f \ar[d]_{i_1} \ar@{}[dr]|(.8){\ulcorner}& B \ar[d] \\ A \ar[r]_-{i_0} & A \times \iso \ar[r] & A \times \iso \coprod_A B}$$ Concretely, $A \times \iso \coprod_A B$ is the unique category with objects $A_0 \coprod B_0$ such that the functor $f\sqcup\id$ to $B$ is fully faithful, and hence a trivial fibration. The bottom composite above is used to define the functorial factorization
\[A \sr{f}{\lra} B \hspace{.5cm} \mapsto \hspace{.5cm} \xymatrix{A \ar[rr]^-{Cf:=i_0} & & A \times \iso \coprod_A B \ar[rr]^-{F_tf:=f\sqcup\id} & &  B}\] Here $A$ does not necessarily inject into the mapping cylinder, because arrows in $A$ that become equal in $B$ get identified, but it is injective on objects; hence $i_0$ is a cofibration.

On morphisms, the functor $C : \Cat^{\2} \to \Cat^{\2}$ sends \[\raisebox{.25in}{\xymatrix{ A \ar[r]^u \ar[d]_f & A' \ar[d]^{f'} \\ B \ar[r]_v & B'}} \hspace{.5cm} \text{to} \raisebox{.25in}{\xymatrix{A \ar[r]^u \ar[d]_{i_0} & A' \ar[d]^{i_0} \\ A \times \iso \coprod_A B \ar[r]_-{u \times \id \sqcup v} & A' \times \iso \coprod_{A'} B'}}\] The counit and comultiplication natural transformations have components \[\vec{\e_f} = \raisebox{.25in}{\xymatrix{ A \ar[d]_{i_0} \ar@{=}[r] & A \ar[d]^f \\ A \times \iso \coprod_A B \ar[r]_-{f\times \id\sqcup\id} & B}} \qquad \vec{\d_f} = \raisebox{.25in}{\xymatrix{ A \ar@{=}[r] \ar[d]_{i_0} & A \ar[d]^{i_0} \\ A \times \iso \coprod_A B \ar[r]_-{i_0\sqcup\id} & A \times \iso \coprod_A A \times \iso \coprod_A B}}\] In particular, $\d_f$ includes $A \times \iso$ into the first copy of this object in the triple pushout; the second copy is not in the image of this map.

Every cofibration in $\Cat$ admits a unique $\C$-coalgebra structure: if $f$ is injective on objects, there is a unique arrow from its codomain to the mapping cylinder so that \[\xymatrix{ A \ar[d]_f \ar[r]^-{i_0} & A \times \iso \coprod_A B \ar[d]^{f\sqcup\id} \\ B \ar@{=}[r] \ar@{-->}[ur] & B} \] commutes. Objects $b \in B$ of the form $b=f(a)$ necessarily map to $(a,0) \in A \times \iso$ while objects not in the image of $A$ necessarily map to themselves in $B$. Because $f\sqcup\id$ is full and faithful, this object map determines the section $B \to A \times \iso \coprod_A B$ on morphisms.  It is easy to check that this lift makes $f$ a $\C$-coalgebra. In particular, the cofibration $j$ is automatically $\I$-cellular, and $\I$ and $\J$ give $\Cat$ an algebraic model structure. 

To show that it is symmetric monoidal, we apply Theorem \ref{existthm} and examine pushout-products of generating (trivial) cofibrations.
 $\I$-cellularity is automatic from the fact that $\Cat$ is a monoidal model category in the ordinary sense \cite{lackhomotopy}, so we must only check that the pushout-product of elements of $\I$ with elements of $\J$ is $\J$-cellular. By an easy computation \[c \htimes j = j \quad \mathrm{and} \quad d \htimes j = e \htimes j = \id_{\2 \times \iso}.\] The first of these has a canonical and the second a unique $\C_t$-coalgebra structure. This defines \[-\htimes- \colon \Calg \times \Ctalg \to \Ctalg. \]

Because each cofibration has a unique $\C$-coalgebra structure, squares with terminal vertex $\Calg$ automatically commute. Because the functors $\Calg \times \Ctalg \to \Ctalg$ and $\Ctalg \times \Calg \to \Ctalg$ are defined symmetrically and $\J$ consists of a single generator, we can apply Theorem \ref{uniquethm} to conclude that \[\xymatrix{ \Ctalg \times \Ctalg \ar[d] \ar[r] & \Ctalg \times \Calg \ar[d] \\ \Calg \times \Ctalg \ar[r] & \Ctalg}\] commutes, proving that the symmetric monoidal algebraic model structure is maximally coherent.
\end{proof}

\begin{thm}\label{ssetmonalgthm} Quillen's original model structure on simplicial sets is a monoidal algebraic model structure with the usual generating (trivial) cofibrations.
\end{thm}
\begin{proof}
It is well-known that simplicial sets form a symmetric monoidal model category generated by the usual sets $\I$ and $\J$ of sphere and horn inclusions. As with $\Cat$, a dimension argument can be used to give an inductive description of the comonad $\C$ in such a way that it is clear that all cofibrations admit unique $\C$-coalgebra structures. In particular, the generators $\J$ are $\I$-cellular, as described explicitly in Example \ref{ex:ssets}, defining an algebraic model structure. 

Because all cofibrations are uniquely cellular, to show that the cartesian product forms an algebraic Quillen two-variable adjunction, we need only worry about the algebraic trivial cofibrations. Here the usual theory of anodyne extensions, a key component of the proof that simplicial sets is a monoidal model category in the non-algebraic sense, is not quite strong enough: the modern proofs show that elements of the pushout-product $\J \htimes \I$ are trivial cofibrations but don't prove that they are $\J$-cellular, that is, that they can be factored as composites of pushouts of the generating horn inclusions $\J$. By Theorem \ref{existthm}, this stronger statement is needed to complete the proof.

However, the classical elementary proof, found for instance in \cite[Theorem 6.9]{maysimplicial}, that if $X$ is a Kan complex and $A$ any simplicial set then the hom-object $X^A$ is Kan shows precisely this. In that argument, $X$ is implicitly regarded as an algebraic fibrant object. Using the implicitly chosen fillers for all horns in $X$, that proof constructs solutions to lifting problems \begin{equation}\label{kanhomlift}\xymatrix{ \Lambda^n_k \ar[d] \ar[r] & X^A \\ \Delta^n \ar@{-->}[ur]}\end{equation} via a combinatorial analysis of the data described by the horn in $X^A$. The given construction has the following architecture. By the calculus of parameterized adjunctions, the lifting problem (\ref{kanhomlift}) corresponds to a lifting problem \[\xymatrix{ & X^{\Delta^n} \ar[d] \\ A \ar[r] \ar@{-->}[ur] & X^{\Lambda^n_k}}\] the solution to which is constructed inductively through an analysis of $(n,m)$-\emph{shuffles}: for each simplex $a\colon \Delta^m \to A$, its image in $X^{\Lambda^n_k}$ is lifted to $X^{\Delta^n}$ by filling an explicit sequence of horns in $X$. Importantly, these lifts are chosen so as to be compatible with previously-specified lifts for all $(m-1)$-simplices of $A$. In other words, the proof inductively specifies solutions to lifting problems \begin{equation}\label{kaninductlift}\xymatrix{ \partial \Delta^m \ar[rr] \ar[d] & & X^{\Delta^n} \ar[d] & \ar@{}[d]|{\leftrightsquigarrow} & \Lambda^n_k \times \Delta^m \displaystyle\po_{\Lambda^n_k \times \partial \Delta^m} \Delta^n \times \partial\Delta^m \ar[d] \ar[r] & X \\ \Delta^m \ar[r] \ar@{-->}[urr] & A \ar[r] & X^{\Lambda^n_k}& & \Delta^n \times \Delta^m \ar@{-->}[ur]} \end{equation} by factoring the displayed element of $\J \htimes \I$ as a composite of pushouts of elements of $\J$. Thus, we see that the inductive step of the proof of \cite[Theorem 6.9]{maysimplicial}, constructing a chosen solution to the lifting problem (\ref{kaninductlift}), establishes the $\J$-cellularity of the maps $\J \htimes \I$.
\end{proof}

This monoidal algebraic model structure on simplicial sets is \emph{mostly} but not \emph{maximally} coherent. It is instructive to see why. It is mostly coherent because all monomorphisms of simplicial sets have a unique $\C$-coalgebra structure, so the lifted functors with codomain $\Calg$ commute because the functors they are lifting commute. 

However, the pushout-product of a pair of generating trivial cofibrations is assigned two different $\C_t$-coalgebra structures, depending on which generator is regarded as a $\C$-coalgebra. We illustrate with an example. Write $h^1_0 \colon \Lambda^1_0 \to \Delta^1$ and $h^2_1 \colon \Lambda^2_1 \to \Delta^2$ for the inclusions of 1- and 2-dimensional horns missing the 0th and 1st faces, respectively. The pushout-product $h^1_0 \htimes h^2_1$ has codomain the solid cylinder $\Delta^1 \times \Delta^2$ and domain a hollow ``trough'' with one of the end triangles and the top square $\Delta^1 \times \Delta^1$ missing. 
\begin{equation}\label{trough}\xymatrix@R=5pt@C=5pt{ \bullet \ar@{}[rrrrd]|{=} \ar@{}[rrrrdd]_(.6){=}  \ar[rrrr] \ar[ddrr]  \ar[dddrrrrrrrr] & & & & \bullet\ar[dddrrrrrrrr]  \ar@{}[ddr]|{=} \\   & & & & \\ & & \bullet \ar[uurr]|(.4){\hole} \ar[dddrrrrrrrr] \ar[rrrrrrd] \ar[rrrrrrrrrrd]|(.1){\hole}|(.2){\hole}|(.3){\hole}|(.4){\hole} &&&& & &  \\ & & & & &  & & & \bullet \ar[ddrr] \ar@{}[ddll]|(.3){=}  \ar@{}[drrrr]|{=}& & & & \bullet \\ & & & & &&&&&&&& & \\ & &&  & & & & & & & \bullet \ar[uurr] & &  }\end{equation}

For simplicial sets, $\C$-coalgebra structures are precisely $\I$-cellular structures, that is, factorizations of a given monomorphism into pushouts of coproducts of elements of $\I$ filtered by attaching degree. The $\I$-cellular structure assigned the horn inclusion $h^n_k$ is given by the factorization \begin{equation}\label{hornicell} \xymatrix{ \Lambda^n_k \ar[r] & \partial \Delta^n \ar[r] & \Delta^n}\end{equation} The first map is a pushout of $\partial \Delta^{n-1} \to \Delta^{n-1}$ and attaches the ``missing face'' to the horn; the second map fills the resulting sphere. 

The following general lemma, stated using the notation relevant to this example, will facilitate our computation. This is an application of the converse of the composition criterion of Theorem \ref{twovarcompthm}.

\begin{lem}\label{polem} Given $i \colon A \to B \in \Ctalg$ and $j \colon J \to K, k \colon K \to L \in \Calg$, the lifted functor $-\htimes-\colon \Ctalg \times \Calg \to \Ctalg$ of a two-variable adjunction of awfs assigns $i \htimes (kj)$ the $\C_t$-coalgebra structure obtained by composing the displayed pushout of the $\C_t$-coalgebra $i \htimes j$ with the $\C_t$-coalgebra $i \htimes k$. \begin{equation}\label{polemeq}\xymatrix@C=30pt{ A \times K \displaystyle\po_{A \times J} B \times J \ar[d]_{i \htimes j} \ar@{}[dr]|(.8){\ulcorner} \ar[r]^{A \times k \sqcup 1} & A \times L \displaystyle\po_{A \times J} B \times J \ar@/^/[ddr]^{i \htimes (kj)} \ar[d]_{p=1 \sqcup_{A \times j} B \times j}  \\ B \times K \ar[r]^-{\iota} \ar@/_/[drr]_{B \times k} & A \times L \displaystyle\po_{A \times K} B \times K \ar[dr]^{i \htimes k} \\ & & B \times L}\end{equation}
\end{lem}
\begin{proof}
It is straightforward to check that (\ref{polemeq}) makes sense, i.e., that the square is a pushout and gives the described factorization of $i \htimes (kj)$. We compute the canonical $\C_t$-coalgebra structure assigned $i \htimes (kj)$ as the composite of these maps and show that it agrees with that assigned $i \htimes(kj)$ by the composition criterion. Write $p$ for the pushout of $i \htimes j$, and write $z_j, z_k, z_p$ respectively for the $\C_t$-coalgebra structures assigned to $i \htimes j$, $i \htimes k$, and $p$. Because $p$ is assigned the coalgebra structure of a pushout, $z_p$ equals \[\xymatrix{ A \times L \displaystyle\po_{A \times K} B \times K \cong \left( A \times L \displaystyle\po_{A \times J} B \times J \right) \displaystyle\po_{{\sim}} B \times K \ar[rrr]^-{C_t p \sqcup R(A \times k \sqcup 1,\iota) \cdot z_j } &  & & Rp.}\] The coalgebra structure assigned the composite is \begin{equation}\label{tentativething}\xymatrix{ B \times L \ar[r]^{z_k} & R(i \htimes k)  \ar[rr]^-{R(R(1,i \htimes k)\cdot z_p),1)} &&  RF(i \htimes (kj)) \ar[r]^{\mu_{i \htimes (kj)}} & R(i \htimes (kj))} \end{equation} By definition, $R(1, i \htimes k) \cdot z_p$ is the top arrow of the lifting problem 
\[\xymatrix{ A \times L \displaystyle\po_{A \times K} B \times K \ar[d]_{i \htimes k} \ar[rrrr]^-{C_t(i \htimes (kj)) \sqcup R(A \times k \sqcup 1, B \times k) \cdot z_j} &&&& R(i \htimes (kj)) \ar[d]^{F(i \htimes (kj))} \\ B \times L \ar@{=}[rrrr] & &&& B \times L}\]
whose canonical solution is the composite (\ref{tentativething}). But this is precisely what is required by the composition criterion of Theorem \ref{compcritthm}, which holds for the lifted functor $i \htimes -$ obtained from a two-variable adjunction of awfs.
\end{proof}

By a similar dimension argument, $\C_t$-coalgebra structures on $\sSet$ are $\J$-cellular structures, that is sequences of monomorphisms which attach fillers for all previously unexamined horns. We use this intuition and the above lemma to compute the coalgebra structures assigned to $h^1_0 \htimes h^2_1$ by the two lifted functors. 

We first apply Lemma \ref{polem} to the $\I$-cellular decomposition (\ref{hornicell}) of $h^1_0$. The pushout-product of $h^2_1$ with the inclusion $\emptyset \to \Delta^1$ is simply $h^1_0$. Hence, the $\J$-coalgebra structure assigned its pushouts, including in particular the first factor of $h^1_0 \htimes h^2_1$ defined in Lemma \ref{polem}, first fills the $\Lambda^2_1$-horn on the front edges of (\ref{trough}) to obtain a ``trough,'' before filling the ``trough'' in the way specified by the lifted functor $\I \times \J \to \Ctalg$. 

On the other hand, the pushout-product of $h^1_0$ with $\partial\Delta^1 \to \Delta^1$ is the monomorphism
\[\xymatrix@R=5pt@C=5pt{ \bullet \ar[dd] \ar[rr]  & & \bullet  \\  \\ \bullet \ar[rr]  &   \ar@{ |-> }[dd] &  \bullet \\  \\ \bullet \ar[ddrr] \ar[dd] \ar@{}[ddr]|{=} \ar[rr] \ar@{}[rrd]|{=} & & \bullet \ar[dd] \\ & & {}  \\ \bullet \ar[rr] & & \bullet  }\] This map has $\J$-cellular structure given by first filling the $\Lambda^2_1$-horn formed by the right and bottom edges and then filling the resulting $\Lambda^2_0$-horn formed by the top edge and the diagonal. Pushouts of this map inherit a similar $\J$-cellular structure. In particular the $\J$-cellular structure assigned $h^1_0 \htimes h^2_1$ by this method first fills the top of the trough (\ref{trough}), at which point it must fill the end triangle very last, using a 3-dimensional horn, not a 2-dimensional one. So this $\C_t$-coalgebra structure can't possibly agree with the one assigned via the other lifted functor. 

\begin{rmk}
There might be multiple ways to define a (cartesian) monoidal algebraic model structure on simplicial sets extending the standard algebraic model structure generated by $\I$ and $\J$. But even if the lifted functor $\I \times \J \to \Ctalg$ were defined by a different procedure than the one described in the proof of Theorem \ref{ssetmonalgthm}, there are no other ways to make these low-dimensional pushout products $\J$-cellular. Hence no monoidal algebraic model structure generated by  $\I$ and $\J$ will be maximally coherent.
\end{rmk}

\begin{rmk}
We expect this sort of argument to apply to many situations, which is why we did not require monoidal algebraic model structures to be maximally coherent.
\end{rmk}

If $(\M, \times, *)$ is a closed monoidal category such that the monoidal unit is terminal, then there is a monoidal product $\wedge$ on $\M_*$ defined as follows. Write $\vee$ for the coproduct in $\M_*$. Given $x \colon * \to X, y \colon * \to Y$ in $\M_*$, the pushout \[\xymatrix@C=40pt{X \vee Y \ar[d] \ar[r]^{ (1 \times y)\vee (x \times 1)} \ar@{}[dr]|(.8){\ulcorner} & X \times Y \ar[d] \\ {*} \ar[r] & X \wedge Y}\] defines a bifunctor $-\wedge - \colon \M_* \times \M_* \to \M_*$ that we call the smash product. The monoidal unit is denoted $S^0 = (*)_+ = * \sqcup *$. See \cite[4.2.9]{hoveymodel}.
 
\begin{thm} If $\M$ is a monoidal algebraic model category and the monoidal unit $*$ is terminal and cofibrant, then $\M_*$ is also a monoidal algebraic model category, symmetric if $\M$ is.
\end{thm}
\begin{proof} 
By what one might call the ``hyper-cube pushout lemma,'' which is an application of the fact that colimits commute with each other, the top square in the cube below is a pushout.
\begin{equation}\label{pushsmash}{\small \xymatrix@R=10pt@C=20pt{ & (A \vee L) \displaystyle\po_{A \vee K} (B \vee K) \ar'[d]^(.6){1}[dd] \ar[rr] \ar[dl] \ar@{}[dr]|(.7){\ulcorner} & & A \times L \displaystyle\po_{A \times K} B \times K \ar[dl] \ar[dd]^{i \htimes j} \ar@{}[dddl]|(.9){\urcorner}& \\ {*} \displaystyle\po_{{*}} {*} \ar[rr] \ar[dd]_{1} & &A \wedge L \displaystyle\po_{A \wedge K} B \wedge K \ar@{-->}[dd]^(.35){i \hsmash j} & & \\ & B \vee L \ar[dl] \ar'[r][rr] && B \times L \ar[dl] & \\ {*} \ar[rr] & & B \wedge L  & & }}\end{equation} The left and bottom faces are pushouts tautologically and definitionally. It follows that the composite rectangle from the top left edge to the bottom right edge is a pushout, and hence that the right face is a pushout. This says that the pushout-smash-product $i \hsmash j$ is a pushout of the pushout-product $i \htimes j$. 

On account of the pullbacks \[\xymatrix{ (\C_t)_*\text{-}\mathbf{coalg} \ar[d] \ar[r] \ar@{}[dr]|(.2){\lrcorner} & \Ctalg \ar[d] \\ (\M_*)^\2 \ar[r]_{U^\2} & \M^\2}\quad\quad \xymatrix{ \C_*\text{-}\mathbf{coalg} \ar[d] \ar[r] \ar@{}[dr]|(.2){\lrcorner} & \Calg \ar[d] \\ (\M_*)^\2 \ar[r]_{U^\2} & \M^\2}\] $(\C_t)_*$-coalgebra or $\C_*$-coalgebra structures for based maps are given by $\C_t$-coalgebra or $\C$-coalgebra structures for the underlying arrows. Hence, we define, e.g., the lifted functor $-\hsmash - \colon (\C_t)_*\text{-}\mathbf{coalg} \times \C_*\text{-}\mathbf{coalg} \to (\C_t)_*\text{-}\mathbf{coalg}$ by assigning $i \hsmash j$ the $\C_t$-coalgebra structure created by the pushout of the $\C_t$-coalgebra $i \htimes j$.

To see that this defines a two-variable adjunction of awfs, we appeal to Theorem \ref{twovarcompthm} and show that this functor satisfies the composition criterion in both variables. This follows easily from the fact that the coalgebra structures assigned to the pushout-smash-products displayed in the front of the diagram below are determined by the coalgebra structures assigned to the pushout-products displayed at the back. By the universal property of the pushouts, the canonical solutions to lifting problems against the front arrows will behave analogously to those against the back arrows; and these, by hypothesis, satisfy the composition criterion.
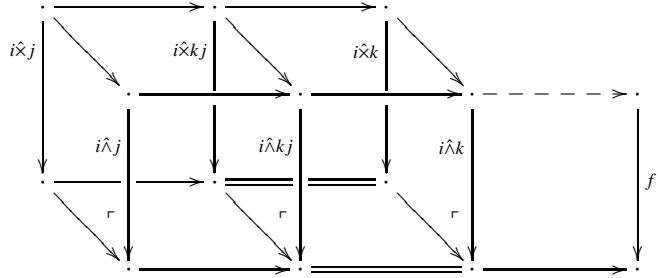
\[\xymatrix{ \cdot \ar[dd]_(.25){i \htimes j} \ar[rr] \ar[dr] \ar@{}[dddr]|(.8){\ulcorner} & & \cdot \ar'[d]_{i \htimes kj}[dd] \ar[rr] \ar[dr] \ar@{}[dddr]|(.8){\ulcorner}& & \cdot \ar'[d]_{i \htimes k}[dd] \ar[dr] \ar@{}[dddr]|(.8){\ulcorner} \\ & \cdot \ar[dd]_(.3){i \hsmash j} \ar[rr] & & \cdot \ar[dd]_(.3){i \hsmash kj} \ar[rr] & & \cdot \ar[dd]_(.3){i \hsmash k} \ar@{-->}[rr] & & \cdot \ar[dd]^f \\ \cdot \ar'[r][rr] \ar[dr] & & \cdot \ar@{=}'[r][rr] \ar[dr] & & \cdot \ar[dr] \\ & \cdot \ar[rr] & & \cdot \ar@{=}[rr] & & \cdot \ar[rr] & & \cdot}\]

Because the monoidal unit $*$ is assumed to be cofibrant and $(-)_+$ is left Quillen, the unit $S^0$ for the monoidal structure on $\M_*$ is cofibrant, and the second condition of Definition \ref{monamsdefn} is automatic. It remains only to see that the algebraic Quillen two-variable adjunction is maximally coherent whenever the original monoidal algebraic model structure is. Because the algebraic model structure on $\M_*$ was defined by pullback, the left-hand square of lifted functors commutes.
\[\xymatrix{(\C_t)_*\text{-}\mathbf{coalg} \times \C_*\text{-}\mathbf{coalg} \ar[d]_{\xi_* \times 1} \ar[r]^-{U \times U} & \Ctalg \times \Calg \ar[d]^{\xi \times 1} \ar[r]^-{\htimes} & \Ctalg \ar[d]^{\xi} \\ \C_*\text{-}\mathbf{coalg} \times \C_*\text{-}\mathbf{coalg} \ar[r]^-{U \times U} & \Calg \times \Calg \ar[r]^-{\htimes} & \Calg }\]
The right-hand square commutes by hypothesis. At each pair of coalgebras in $(\M_*)^\2$, the $(\C_t)_*$-coalgebra structure assigned their pushout-smash-product is determined by the $\C_t$-coalgebra structure assigned the pushout of the arrow in their image along the top row of this diagram; its $\C_*$-coalgebra structure is similarly determined by the $\C$-coalgebra structure assigned the pushout of the map in the image at the bottom right. Writing down explicit formulae (I.5.4), it is easy to see that the process of assigning coalgebra structures to pushouts commutes with the comparison map for $\M$. 
\end{proof}

\begin{thank} This work benefited from conversations with Peter May and Richard Garner. The author is grateful for an opportunity to present this work to the group at Sheffield University at the invitation of Eugenia Cheng, greatly aiding the expository process. She continually appreciates the supportive environment provided by the algebraic topology and category theory seminar at the University of Chicago. An anonymous referee made a number of cogent suggestions leading to structural improvements. Finally, the author is grateful for financial support from the NSF Graduate Research Fellowship.
\end{thank}

\bibliographystyle{elsarticle-num}

\end{document}